\documentclass[a4paper,11pt,reqno]{amsart}
\usepackage{enumitem}
\usepackage{amssymb, amsmath}
\usepackage[margin=2cm]{geometry}
\usepackage{mathrsfs}%,dsfont}
\usepackage{amscd} \usepackage[active]{srcltx} \usepackage{verbatim}
\usepackage[colorlinks,linkcolor={blue},citecolor={blue},urlcolor={black}]{hyperref}
\usepackage[utf8]{inputenc}
\usepackage{wrapfig}

 \usepackage{tikz}
\usetikzlibrary{calc,positioning,decorations.pathreplacing}
\usetikzlibrary{backgrounds}
\usetikzlibrary{decorations.pathmorphing}
\usetikzlibrary{arrows}
\usetikzlibrary{shapes.geometric}
\usetikzlibrary{calc}
\usetikzlibrary {arrows.meta,bending,positioning}

\theoremstyle{plain}
\newtheorem{theorem}{Theorem}
\newtheorem{lemma}{Lemma}[section]

\newtheorem{proposition}[lemma]{Proposition}

\newtheorem{assumption}[lemma]{Assumption}

\newtheorem*{definition*}{Definition}

\theoremstyle{remark}
\newtheorem{remark}[lemma]{Remark}

\newtheorem*{claim*}{Claim}
\newtheorem*{remark*}{Remark}
\newtheorem*{example*}{Example}
\newtheorem*{notation*}{Notation}
\newcommand{\step}[1]{\paragraph{\textbf{Step #1}}}

\numberwithin{equation}{section}

\def\N{{\mathbb N}}

\renewcommand{\d}{\mathfrak{d}}

\renewcommand{\phi}{\varphi}

\newcommand{\TTd}{{\mathbb{T}^d}}

\newcommand{\e}{\epsilon}
\newcommand{\TND}{\mathbb{T}^d_N}
\newcommand{\PP}{\mathbb{P}}
\newcommand{\cH}{\mathcal{H}}

\newcommand{\cD}{\mathcal{D}}

\newcommand{\cL}{\mathcal{L}}

\newcommand{\RRd}{\mathbb{R}^d}

\newcommand{\RR}{\mathbb{R}}

\newcommand{\cF}{\mathcal{F}}

\newcommand{\cP}{\mathcal{P}}

\newcommand{\hs}{\ensuremath{\hspace{1cm}}}
\newcommand{\h}{\ensuremath{\hspace{0.1cm}}}
\newcommand{\EE}{\ensuremath{\mathbb{E}}}

\newcommand{\Q}{\ensuremath{\mathcal{Q}}}

\newcommand{\tf}{{t_{\mathrm{fin}}}}

\usepackage{fancyhdr}
\makeatletter

\title{The Porous Medium Equation: Multiscale Integrability in Large Deviations}
\author{Benjamin Gess}
\author{Daniel Heydecker}
\address{B. Gess: Institut f\"ur Mathematik
Technische Universit\"at Berlin
Stra{\ss}e des 17. Juni 136
10623 Berlin}
\email{benjamin.gess@tu-berlin.de}
\address{D. Heydecker: University of Oslo, Blindern, 0316 Oslo}
\email{daniehey@math.uio.no}



\makeatother
\begin{document}

\begin{abstract}
We consider a zero-range process $\eta^N_t(x)$ with superlinear local jump rate, which in a hydrodynamic-small particle rescaling converges to the porous medium equation $\partial_t u=\frac12\Delta u^\alpha, \alpha>1$. As a main result we obtain a large deviation principle in any scaling regime of vanishing particle size $\chi_N\to 0$. The key challenge is to develop uniform integrability estimate on the nonlinearity $(\eta^N(x))^\alpha$ in a situation where neither pathwise regularity nor Dirichlet-form based regularity is readily available. We resolve this by introducing a novel multiscale argument exploiting the appearance of {\em pathwise regularity across scales}.   \end{abstract}

\maketitle

\section{Introduction} \label{sec: intro} The problem of finding the large deviations of gradient particle systems is a classical problem \cite{kipnis1989hydrodynamics,kipnis1990large,kipnis1998scaling}, which has now been resolved for a number of models with bounded and nondegenerate diffusivity. Once one allows unbounded and degenerate diffusivity, many open questions in large deviations remain. \\ \\ The goal of this paper will be to introduce a general and robust analysis to understand where these features cause difficulties, and how they may be resolved, for a simple and paradigmatic example with both unbounded and degenerate diffusion. We will study a zero-range process $\eta^N_\bullet$ which converges, with a further rescaling introduced in \cite{gess2023rescaled}, to the porous medium equation (PME) \begin{equation}
				\label{eq: PME} \partial_t u_t=\frac12\Delta(u_t^\alpha),\qquad x\in \TTd, t\in [0,\tf]; \qquad \alpha>1.
			\end{equation} The appearance of the degenerate diffusion at the level set $u=0$ is made possible by performing a hydrodynamic rescaling of the zero-range process with local jump rate $g(k)=dk^\alpha$ on the discrete torus, so that the rescaled spatial variable takes values in $\TND:=\{0, \frac{1}{N}, \dots, \frac{N-1}{N}\}^d$, and performing an additional rescaling in both time and particle size by a factor $\chi_N \to 0$. A full definition of the rescaling will be given in Section \ref{sec: prelim}. \\\\ Previous literature \cite{gess2023rescaled,fehrman2019large} crucially  relied on a particular relative scaling regime of the parameters $N, \chi_N$, requiring particle size $\chi_N$ to decay sufficiently quickly: As will be discussed below, the properties of the particle system in the opposite scaling regime of large particles are fundamentally different, and established techniques are no longer available. The contribution of this work is to resolve the challenges from degenerate and unbounded diffusivity {\em without any such scaling restriction}, thereby extending the large deviation principle found in \cite[Theorem 2]{gess2023rescaled} to the whole scaling regime. \begin{theorem}\label{thrm: LDP} Let $\chi_N \to 0$. Fix $\rho\in C(\mathbb{T}^d, (0,\infty))$, and let $(\Omega, \cF,\PP)$ be a probability space on which are defined, for all $N$, a rescaled zero-range process $\eta^N_\bullet$ with parameter $\alpha\ge 1$ and initial data distributed according to the slowly varying local equilibrium distribution $\rm{Law}_{\PP}[\eta^N_0]= \Pi^N_\rho$ given by \eqref{eq: svleq}. Then $\eta^N_\bullet$ satisfy a large deviations principle in the Skorokhod space\footnote{See Section \ref{sec: prelim}.} $\mathbb{D}$ with speed $N^d/\chi_N$ and the same rate function $\mathcal{I}_\rho$ as defined in \cite{gess2023rescaled}. 	\end{theorem}  Independently of the scaling regime, one must contend with the same key challenges of proving a {\em superexponential estimate} and a {\em uniform integrability estimate}, as detailed in Section \ref{sec: theorem statements}. However, the properties of the process, and hence the possible techniques, are fundamentally different, depending on whether $\chi_N\to 0$ sufficiently fast, or sufficiently slowly, with $N$. Indeed, in \cite{gess2023rescaled}, the key results are derived from estimates using the supermartingale \begin{equation}
		\label{eq: superexp ent diss} \exp\left(\frac{N^d}{\chi_N}\left(\cH(\eta^N_t)-\cH(\eta^N_0)+\int_0^t \|\nabla_N(\eta^N_s)^{\alpha/2}\|_{L^2_x(\TND)}^2 ds - Ct\chi_N^\delta N^2 \right)\right) \end{equation} for some fixed $C<\infty, \delta>0$, and where $\cH$ is the Boltzmann entropy. The usefulness of such an estimate relies crucially on the scaling $\chi_N^\delta N^2\lesssim 1$, which characterises the regime in which the variance of each discrete gradient $\nabla_N(\eta^N_t)^{\alpha/2}(x)$ remains bounded: Outside this scaling regime, such an estimate is impossible due to variance blowup. The equivalent relative scaling is also used in \cite{fehrman2019large} for the Stratonovich-Dean-Kawasaki equation \begin{equation}
	\label{eq: SPDE} \partial_t u^\e_t = \frac12 \Delta \left((u^\epsilon_t)^\alpha\right)-\sqrt{\e} \nabla\cdot\left((u^\e_t)^{\alpha/2}\circ \xi^\delta\right)
\end{equation} driven by the convolution $\xi^\delta$ of a white noise on a on a scale $\delta=\delta(\epsilon) \to 0$, where the scaling $\epsilon \delta(\epsilon)^{-(d+2)} \to 0$ is required to produce a martingale like \eqref{eq: superexp ent diss} with the continuum entropy dissipation $\|\nabla (u^\epsilon)^{\alpha/2}\|_{L^2_x}.$  In both cases, the pathwise regularity relies on the smallness of the noise intensity $\chi_N, \epsilon$ with respect to the correlation length $N^{-1}, \delta$, and the key results may be derived from this regularity. \\\\ On the other hand, the stronger noise improves the regularity in probability space enjoyed by laws, as measured by the Dirichlet form  \begin{equation}
				\label{eq: dirichlet on laws} \d_N(f_N):=\int_{X_N}\frac12 \sum_{x\sim y}  \h \frac{(\eta^N(x))^\alpha}{\chi_N}\left(\sqrt{f_N(\eta^{N,x,y})}-\sqrt{f_N(\eta^N)}\right)^2 \Pi^N_a(d\eta^N)
			\end{equation} on probability densities $f_N$ with respect to the invariant measures $\Pi^N_a$, since the smallest nonzero jump rate $\sim \chi_N^{\alpha}$ is then less degenerate. Techniques based on estimating expectations from the Dirichlet form are used in \cite{kipnis1989hydrodynamics,kipnis1998scaling}, which thus benefit from the stronger noise, in the scaling regime (with the current notation) $\chi_N\sim 1$. In order to deal with any scaling regime, we are forced to contend with the problems of integrability, required to start a so-called superexponential estimate, and the superexponential estimate itself, while being able to access neither of these kinds of regularity in any direct way. \\ The resolution which we develop is to bridge these two kinds of regularity. Given a box $\widetilde{B}$ of side length $\widetilde{l}_N$, and a partition $\cP|_{\widetilde{B}}$ of $\widetilde{B}$ by boxes $B$ of side length $l_N\ll \widetilde{l}_N$, we show, in Lemma \ref{lemma: reg by paths}, that coarse-grained quantities $\Lambda_B \eta^N$ given by \eqref{eq: def Lambda} satisfy the regularity estimate, for translationally invariant distributions\footnote{See Lemma \ref{lemma: gallilean pdfs}.} $f_N$, \begin{equation}\begin{split}&\mathbb{E}_{f_N\Pi^N_a}\left[\left\|\nabla_{\widetilde{l}_N/l_N}(\Lambda_B \eta^N)^{\alpha/2}\right\|_{\ell^2(B\in \cP|_{\widetilde{B}})}^2\right] \\ &\hspace{4cm} \le C\widetilde{l}_N^2\chi_N N^{-d}\d_N(f_N)\\& \hspace{4cm}+C\chi_N^\delta l_N^{-d/2}\left(\frac{\widetilde{l}_N}{l_N}\right)^2\mathbb{E}_{f_N\Pi^N_a}\left(1+\|(\eta^N)^\alpha\|_{\ell^1(\TND)}\right). \end{split}\label{eq: preintro mesoscopic est} \end{equation}  \setlength{\columnsep}{1em}              % 

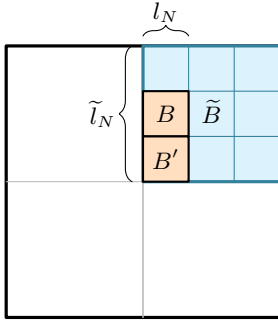
\begin{wrapfigure}[13]{l}{5cm}
  \vspace{-0.8cm}\centering
  \begin{tikzpicture}[scale=0.60, line join=round, line cap=round, font=\footnotesize]
    % --- Parameters (picture-only) ---
    \def\L{6}         % total side length in picture coordinates
    \def\N{2}         % coarse boxes per side (2x2)
    \pgfmathsetmacro\outer{\L/\N}

    % Outer boundary + coarse grid
    \draw[very thick] (0,0) rectangle (\L,\L);
    \draw[gray!55] (\outer,0) -- (\outer,\L);
    \draw[gray!55] (\outer,0) -- (\outer,\L);
    %\draw[gray!55] (2*\outer,0) -- (2*\outer,\L);
    \draw[gray!55] (0,\outer) -- (\L,\outer);
    %\draw[gray!55] (0,2*\outer) -- (\L,2*\outer);

    % Special top-right coarse box: \widetilde{B}
    \coordinate (tBL) at (\outer,\outer);
    \coordinate (tTR) at (\L,\L);

    % distinguish "tilde" objects by color
    \fill[cyan!12] (tBL) rectangle (tTR);
    \draw[very thick, cyan!55!black] (tBL) rectangle (tTR);

    % Subpartition of \widetilde{B} into 2x2 fine boxes
    \def\m{3}
    \pgfmathsetmacro\small{\outer/\m}
    \draw[cyan!55!black] ($(tBL)+(\small,0)$) -- ($(tBL)+(\small,\outer)$);
    \draw[cyan!55!black] ($(tBL)+(2*\small,0)$) -- ($(tBL)+(2*\small,\outer)$);
    \draw[cyan!55!black] ($(tBL)+(0,\small)$) -- ($(tBL)+(\outer,\small)$);
    \draw[cyan!55!black] ($(tBL)+(0,2*\small)$) -- ($(tBL)+(\outer,2*\small)$);

    % Highlight and label two fine boxes B and B'
    % B: top-left fine box in \widetilde{B}
    \coordinate (Bbl) at ($(tBL)+(0,\small)$);
    \coordinate (Btr) at ($(tBL)+(\small,2*\small)$);
    \fill[orange!25] (Bbl) rectangle (Btr);
    \draw[thick] (Bbl) rectangle (Btr);
    \node[fill=none, inner sep=1.2pt] at ($(Bbl)!0.5!(Btr)$) {$B$};

    % B': bottom-right fine box in \widetilde{B}
    \coordinate (Bpbl) at ($(tBL)+(0,0)$);
    \coordinate (Bptr) at ($(tBL)+(\small,\small)$);
    \fill[orange!25] (Bpbl) rectangle (Bptr);
    \draw[thick] (Bpbl) rectangle (Bptr);
    \node[fill=none, inner sep=1.2pt] at ($(Bpbl)!0.5!(Bptr)$) {$B'$};

    % Put \widetilde{B} label OUTSIDE with an arrow so it can't be obscured
    \node[anchor=west, fill=none, inner sep=1.2pt] (tLab) at ($(tTR)+(-1.8*\small,-1.45*\small)$) {$\widetilde{B}$};

    % Length annotations
    % Side length of \widetilde{B}: \widetilde{l}_N (brace to the left)
    \draw[decorate,decoration={brace,amplitude=5pt}]
      ($(tBL)+(-0.24,0)$) -- node[midway,xshift=-11pt, fill=white, inner sep=1pt] {$\widetilde{l}_N$}
      ($(tBL)+(-0.24,\outer)$);

    % Side length of a fine box: l_N (brace above the top edge of B)
    \coordinate (BtopL) at ($(tBL)+(0,3*\small)$);
    \coordinate (BtopR) at ($(tBL)+(\small,3*\small)$);
    \draw[decorate,decoration={brace,amplitude=4pt}]
      ($(BtopL)+(0,0.22)$) -- node[midway,yshift=9pt, fill=none, inner sep=1pt] {$l_N$}
      ($(BtopR)+(0,0.22)$);
  \end{tikzpicture}

  \caption{Two steps of the scheme.}
  \label{fig:coarse-fine-box}
\end{wrapfigure}In addition to linking the pathwise regularity to the regularity of laws, this estimate gains two compensations relative to \eqref{eq: superexp ent diss}: the lattice size $N$ is replaced by the effective lattice size $(\widetilde{l}_N/l_N)$ in the error term, and the coarse-graining scales down the noise intensity $\chi_N$ by a factor of $l_N^{-d/2}$.  
   \\\\In order to allow any rate of convergence $\chi_N\to 0$, this estimate can only be used when the scales $l_N\ll \widetilde{l}_N$ are constructed to keep the error term under control. As a result, intermediate scales are needed to go from the microscopic to the fully macroscopic. Theorem \ref{thrm: main integrability estimate} is achieved by patching together as many such scales as are necessary in a {\em multiscale approach}, propagating the gain of integrability from macroscopic to meso- and ultimately microscopic scales. Two steps of this scheme are displayed in Figure \ref{fig:coarse-fine-box}: The whole space is partitioned into boxes $\widetilde{B}$ of side length $\widetilde{l}_N$, each of which is partitioned in turn into boxes of side length $l_N$. \subsection{Statement of Results} \label{sec: theorem statements}
   Since the general pattern for establishing a result such as Theorem \ref{thrm: LDP} is by now well-known and classical \cite{benois1995large, kipnis1998scaling}, the emphasis of the paper will be on those parts of the arguments where either the unboundedness or degeneracy of the local jump rate demands new ideas. In attempting to follow the general framework, the first technical ingredient one needs is  a \emph{superexponential estimate}, which allows us to deduce convergence of functionals of the form $ N^{-d}\sum_{x\in \TND} H(x)(\eta^N(x))^\alpha$ from convergence of the configuration $\eta^N_\bullet$ in a topology like $C_t(L^1_x)_{\rm w}$. We write  $\overline{\eta}^{N,N\epsilon}(x)$ for the averages of a configuration $\eta^N$ on a scale $N\epsilon$: \begin{equation}\label{eq: eta N L} \overline{\eta}^{N,N\epsilon}(x):=\frac{1}{(2\lfloor \epsilon N\rfloor +1)^d} \sum_{y\in B_{\lfloor \epsilon N\rfloor}(x)} \eta^N(y)\end{equation} where the average is over the lattice box of microscopic side length $2\lfloor \epsilon N\rfloor +1$ centred at $x$. The problem of taking limits of the nonlinearity is resolved by the following theorem, which allows us to replace the nonlinearity with the same nonlinearity applied to local averages, up to a small error. \begin{theorem}
	\label{thrm: supex}[Superexponential Estimate for the Diffusivity] Fix $a\in (0,\infty)$, and let $\eta^N_\bullet$ be as in Theorem \ref{thrm: LDP} with initial data $\eta^N_0\sim \Pi^N_a$. Then, for all $\delta>0$, \begin{equation} \label{eq: supex conclusion}
	 \limsup_{\epsilon\to 0} \limsup_{N\to \infty} \frac{\chi_N}{N^d} \log \mathbb{P}_{\Pi^N_a}\left(\int_0^\tf \frac1{N^d}\sum_{x\in \TND} V^N_\alpha(\eta^N_t,\epsilon,x) dt >\delta \right)=-\infty
	\end{equation} where $V^N_\alpha$ is the error given by  \begin{equation}
		\label{eq: VNA} V^N_\alpha(\eta^N,\epsilon,x):=\left|(\eta^N(x))^\alpha - (\overline{\eta}^{N,N\epsilon}(x))^\alpha\right|.
	\end{equation}
\end{theorem} It is this estimate in which both degenerate and superlinear diffusion threaten to derail the usual methods of `one- and two-block estimates' \cite{kipnis1989hydrodynamics,benois1995large,kipnis1998scaling}. These arguments rely on the rapid convergence to equilibrium on large microscopic boxes, possibly separated by a small macroscopic distance. For the rescaled zero-range process, the Dirichlet form \eqref{eq: dirichlet on laws} will see all of the jump rates, including the smallest nonzero rate $\chi^{\alpha-1}_N\to 0$, and the compactness used in the classical argument may be lost as the jump rate degenerates. For this reason, the usual technique cannot be used: To replace it, we develop an argument exploiting pathwise-type estimates similar to \eqref{eq: preintro mesoscopic est} in Section \ref{sec: supex}.  \\ \\   The appearance of a superlinearity $u\mapsto u^\alpha$ in \eqref{eq: VNA} adds a separate fundamental difficulty, which may be seen as follows. While $\overline{\eta}^{N, N\epsilon}$ necessarily remains bounded as $N\to \infty$, potentially with disadvantageous dependence in $\epsilon$, it is \emph{a priori} possible that, on some sequence of events of probability $\exp(-zN^d/\chi_N)$ for fixed $z<\infty$, that one sees a divergence of the $L^\alpha_{t,x}$-norm $$\int_0^\tf \frac1{N^d}\sum_{x\in \TND} (\eta^N_t(x))^\alpha dt\to \infty. $$ The same would thus hold for $V^N_\alpha$ for every $\epsilon$, which would thus exclude a smallness condition such as \eqref{eq: supex conclusion}. 
In fact, something even slightly stronger is needed. An examination of how the conclusion \eqref{eq: supex conclusion} is obtained from the one- and two- block estimates shows that we must replace both instances of $u\mapsto u^\alpha$ in \eqref{eq: VNA} with a Lipschitz nonlinearity in order to complete the theorem. This substitution will introduce a further error, and quantifying this error demands a \emph{uniform} integrability estimate on $(\eta^N_t(x))^\alpha$, in order  that the contributions to $\int_0^T N^{-d}\sum_xV^N_\alpha(\eta^N_t,\epsilon,x) dt$ from the time-space region $\{\eta^N_t(x)>M\}$ are small, with superexponentially high probability, in the limit $M\to \infty$. For $b>0$, we set $X_{N,b}$ to be the set of all configurations with the mass bound \begin{equation}\label{eq: def XNB} X_{N,b}:=\left\{\eta^N: \frac1{N^d}\sum_{x\in \TND} \eta^N(x)\le b\right\}. \end{equation} With this defined, the requisite integrability is provided by the following estimate. \begin{theorem}
	\label{thrm: main integrability estimate}[Uniform Integrability of the Diffusivity] Let $\chi_N\to 0$ be any sequence and fix $a, b>0$.  Then, in the notation of Theorem \ref{thrm: supex}, there exists a constant $\beta>0$ and a strict Young function $\Phi$, depending only on the sequence $(\chi_N)_{N\ge 1}$, such that \begin{equation} \label{eq: main integrability conclusion} \limsup_N \frac{\chi_N}{N^d}\log \mathbb{E}_{\Pi^N_a}\left[\exp\left(\frac{\beta N^d}{\chi_N}\int_0^\tf \|(\eta^N_s)^\alpha\|_{\ell_\Phi} 1(\eta^N_s\in X_{N,b})ds \right)\right]<\infty.\end{equation} 
\end{theorem} Here, $\|\cdot\|_{\ell_\Phi}$ denotes the Orlicz norm associated to the Young function $\Phi$, whose definition will be recalled in \eqref{eq: orlicz} below. As will be recalled, the \emph{strictness} of a Young function means that $\Phi(u)/u\to \infty$ as $u\to \infty$, so that the corresponding Orlicz norm penalises large values strictly more strongly than the corresponding $L^1$ norm and allows a suitable interpolation estimate. In particular, Theorem \ref{thrm: main integrability estimate} expresses the uniform integrability motivated above. \\ \\ A further technical aspect, where the superlinear growth plays a nontrivial role, is the exponential tightness, which allows the use of compactness methods in the topology of $\mathbb{D}$ in both upper and lower bounds. Attempting to reproduce the usual argument \cite{benois1995large,kipnis1998scaling}, we find that we must show that quantities like $$ \int_s^tN^{-d}\sum_{x\in \TND} \Delta H(x) \eta^N_u(x)^\alpha du  $$ are small, uniformly in $|s-t|$, for any given test function $H\in C^2(\TTd)$. This is the same requirement of uniform integrability discussed above Theorem \ref{thrm: main integrability estimate}, and as an application of the theorem we obtain \begin{proposition} \label{prop: tightness}
	For every $a\in (0,\infty)$, the processes $\eta^N_\bullet$ are exponentially tight in the topology of $\mathbb{D}$ when the initial data $\eta^N_0$ is sampled from $\Pi^N_a$. That is, for every $z<\infty$, there exists a compact $\mathcal{K}\subset\mathbb{D}$ such that \begin{equation}
		\label{eq: ET conclusion} \limsup_{N\to \infty} \frac{\chi_N}{N^d}\log \mathbb{P}_{\Pi^N_a}\left(\eta^N_\bullet \not \in \mathcal{K}\right)\le -z.
	\end{equation} \end{proposition} Once these technical results are in place, Theorem \ref{thrm: LDP} follows. Theorems \ref{thrm: supex} - \ref{thrm: main integrability estimate} and Proposition \ref{prop: tightness} extend from the equilibrium measure $\Pi^N_a$ to the non-equilibrium setting of Theorem \ref{thrm: LDP} as in \cite[Theorem 3.2]{kipnis1989hydrodynamics}, \cite[Lemma 4.5]{quastel1999large}, see also \cite[Section 10]{gess2023rescaled}. The arguments of \cite{benois1995large, kipnis1998scaling} produce large deviation estimates with potentially different rate functions $\mathcal{I}^{\rm up}, \mathcal{I}^{\rm lo}$, see also \cite[Sections 9 - 10]{gess2023rescaled} for some technical aspects due to the superlinear growth of the nonlinearity $u\mapsto u^\alpha$ and \cite{fehrman2019large,gess2023rescaled,heydecker2023large,fgh2025} for a discussion of the problem of obtaining matching upper and lower bounds. The upper bound is strengthened by showing that $\mathcal{I}_\rho$ may be taken to be infinite unless certain entropy dissipation $\int_0^\tf \cD_\alpha(u_s)ds$ is finite, see \cite[Chapter 5, Theorem 7.1]{kipnis1998scaling} and \cite{fgh2025} for a simplified proof. For the lower bound,  \cite[Theorem 39]{fehrman2019large} identifies the rate function $\mathcal{I}_\rho$ with the rate function $\mathcal{I}^{\rm lo}$, which is defined a priori as the lower semicontinuous envelope of $\mathcal{I}_\rho$ restricted to a set of sufficiently regular trajectories. Together, these resolve the problem of matching the upper and lower bounds and conclude the large deviation principle \cite[Equations (1.9 - 1.10)]{gess2023rescaled} asserted in Theorem \ref{thrm: LDP}. 
	
    \subsection{A Sketch Proof of Theorem \ref{thrm: main integrability estimate}} \label{sec: sketch} Having explained the key difficulties posed by degenerate and superlinear diffusion, we now unveil the key ideas in their resolution. In the previous work \cite{gess2023rescaled}, and works on related SPDEs \cite{fehrman2019large,fehrman2021well}, both the integrability and the superexponential estimate were treated by proving a pathwise regularity result, leveraging the fundamental regularity of the {\em skeleton equation} defining the rate function $\mathcal{I}_\rho$: \begin{equation}
		\label{eq: sk} \partial_t u_t=\frac12\Delta (u_t^\alpha)-\nabla\cdot(u_t^{\alpha/2} g); \qquad g\in L^2([0,\tf]\times\TTd, \RRd).
	\end{equation} This equation was analysed in \cite{fehrman2019large}; the fundamental regularity, for which an estimate can be obtained, is the time-integrated entropy dissipation, given by the Sobolev norm $\|\nabla u^{\alpha/2}\|_{L^2_{t,x}}^2<\infty$. As discussed below \eqref{eq: superexp ent diss}, leveraging this regularity directly at the level of the particle system or SPDE requires a particular scaling relation on $\chi_N, N$, and we must therefore argue differently. \\ \\ The  integrability estimate Theorem \ref{thrm: main integrability estimate} turns out to be the most challenging of the steps towards Theorem \ref{thrm: LDP}, and is therefore the main focus of the paper. Several plausible strategies for such an estimate leverage properties of the PME \eqref{eq: PME}, such as the `coming down from infinity', or attempting to close $L^\alpha_x$- or $H^{-1}_x$-estimates \cite{vazquez2007porous}. Such methods are, at least for large deviations and in general dimension, doomed to failure from the outset. The quantities which one finds in the course of such an estimate are stronger than those which can be estimated from \eqref{eq: sk}, and closing such an estimate would ultimately be incompatible with the rate function $\mathcal{I}_\rho$ for which we aim. For this reason, it is the intrinsic regularity of \eqref{eq: sk}, and not \eqref{eq: PME}, which must dictate the strategy. \\
		The most na\"ive approach to leveraging the regularity discussed above is to use lattice discretisation of the Sobolev embeddings and the discrete regularity obtained from \eqref{eq: superexp ent diss} to obtain a discrete integrability estimate. This is the strategy carried out in \cite[Sections 5-7]{gess2023rescaled}, but requires the same scaling relation used therein. One finds, whether arguing at the pathwise level as in \cite{gess2023rescaled}, or through the Feynmann-Kac formula as in Lemma \ref{lemma: reg by paths} below, an estimate of the form \begin{equation} \begin{split}\label{eq: where the scaling relation comes from} \mathbb{E}_{f_N\Pi^N_a}\left[\left\|\nabla_N {(\eta^N)^{\alpha/2}}\right\|^2_{L^2_x(\TND)}\right]  & \le C\chi_N N^{2-d}\mathfrak{d}_N(f_N) \\ &\hs +C\chi_N^\delta N^2\mathbb{E}_{f_N\Pi^N_a}\left(1+\|(\eta^N)^\alpha\|_{\ell^1(\TND)}\right) \end{split} \end{equation} for some (fixed) $\delta\in (0,1]$, where $f_N$ is any probability density with respect to $\Pi^N_a$ and where $\d_N$ is the Dirichlet form associated to the dynamics. The first term is of the correct form to deduce the desired estimates from the Feynman-Kac formula, see, for example, the proof of \cite[Theorem 2.1]{kipnis1989hydrodynamics}), and \eqref{eq: FK formula} below. However, the factor $\chi_N^\delta N^2$ in the error term means that the error term on the second line diverges, unless the same scaling relation $\chi_N^\delta N^2\lesssim 1$ as in \cite{gess2023rescaled} is imposed. The exponent $\delta\le 1$ arises from from controlling certain error terms close to zero in the case $\alpha\in (1,2]$, but in fact the estimate cannot be improved beyond setting $\delta=1$, corresponding to variance terms $N^2 {\rm Var}_{\Pi^N_a}(\eta^N(x)-\eta^N(y))\sim N^2\chi_N$. For this reason, we see the qualitatively different behaviour of variance blowup in the regime $\chi_N N^2\to \infty$, destroying the pathwise regularity \eqref{eq: superexp ent diss}, and anticipate this threshold being the fundamental one.  \\  The three key insights of Theorem \ref{thrm: main integrability estimate} are that we may avoid any scaling relationship between $\chi_N$ and $N$, because \begin{enumerate}[label=\roman*).] 
			\item the error term gains an advantageous scaling when replacing $\eta^N$ by a local average on scale $l_N\gg 1$, which permits measuring the regularity on some larger scale $l_N\ll \widetilde{l}_N\le N$; \item such estimates imply a gain of integrability across scales $l_N, \widetilde{l}_N$, which may be averaged over boxes $\widetilde{B}$ partitioning $\TND$; and \item the desired estimate may be obtained by applying (ii) to a sequence of scales $1=l^1 \ll l^2 \dots \ll l^K=N$.  
		\end{enumerate}  We now expound these points in more detail. In the following, we outline the strategy of Theorem \ref{thrm: main integrability estimate} based on the estimate \eqref{eq: preintro mesoscopic est} above. For the clarity of exposition, we will temporarily engage in the two simplifying assumptions that \begin{assumption}\label{assumptions sketch}
			\begin{enumerate}[label=\alph*)] \item $\chi_N\lesssim N^{-\gamma}$, for some $\gamma>0$, and
			\item $N=2^m$ is a dyadic integer, for some $m\ge 1$.
		\end{enumerate}
		\end{assumption} \step{1. Regularity across scales} The first point is the deduction of the estimate displayed at \eqref{eq: preintro mesoscopic est} above, of which we now give more details. We recall that we have a large box $\widetilde{B}$ of side length $\widetilde{l}_N$, which is partitioned by $\cP|_{\widetilde{B}}$ into boxes $B$ of side length $l_N$. The coarse-graining is defined by \begin{equation}\label{eq: def Lambda} \Lambda_B\eta^N:=\left(\frac{1}{|B|}\sum_{x\in B}(\eta^N(x))^\alpha\right)^{1/\alpha} \end{equation} and the lattice Sobolev norm estimated from the Dirichlet form is $$\left\|\nabla_{\widetilde{l}_N/l_N}(\Lambda_B \eta^N)^{\alpha/2}\right\|_{\ell^2(B\in \cP|_{\widetilde{B}})}^2:=\left(\frac{l_N}{\widetilde{l}_N}\right)^{d-2}\sum_{B\sim B'} ((\Lambda_B \eta^N)^{\alpha/2}-(\Lambda_{B'}\eta^N)^{\alpha/2})^2.$$ The key gain from equation \eqref{eq: preintro mesoscopic est} is that, given any scale $l_N\gg 1$ and any rate of decay $\chi_N\to 0$, we may find a larger scale $\widetilde{l}_N\gg l_N$ so that the error term remains under control. \\ \step{2. Gain of Integrability across scales} The second key point concerns the deduction of integrability statements on the smaller scale. If one were to na\"ively first postprocess the small-scale regularity \eqref{eq: preintro mesoscopic est} to the full lattice, the error term would gain a factor of $(N/\widetilde{l}_N)^2$, undoing the compensation produced by the effective side length. The strategy is instead to first deduce integrability, and then to postprocess the integrability into global statements. Define, for $p>1$ to be chosen later, \begin{equation} \label{eq: Delta toy} \Delta_{\widetilde{B}}(\eta^N):=\left[\|(\Lambda_B \eta^N)^\alpha\|_{\ell^p(B\in \cP|_{\widetilde{B}})}-2(\Lambda_{\widetilde{B}}\eta^N)^\alpha\right]_+\end{equation} where the $\ell^p$ norm is with respect to the uniform measure on $\cP|_{\widetilde{B}}$, and here, and throughout the paper, we write $\ell^p(B\in \cP)$ (and similar expressions) to emphasise which variable is being summed. A discrete Sobolev-Gagliardo-Nirenberg inequality, see Proposition \ref{prop: discrete Sobolev}, and \eqref{eq: preintro mesoscopic est} show that, for $p>1$ small enough depending on the dimension,\begin{equation}
			\label{eq: sobolev embedding toy} \begin{split}\mathbb{E}_{f_N \Pi^N_a}\left[ \Delta_{\widetilde{B}}(\eta^N)\right] & \le  C\widetilde{l}_N^2\chi_N N^{-d} \d_N(f_N)  \\ & \hs + Cl_N^{-d/2} \left(\frac{\widetilde{l}_N}{l_N}\right)^2\chi_N^\delta\mathbb{E}_{f_N\Pi^N_a}\left(1+\|(\eta^N)^\alpha\|_{\ell^1(\TND)}\right). \end{split} 
		\end{equation} It is this estimate which may be postprocessed to the full scale without undoing the compensation in \eqref{eq: preintro mesoscopic est}. We now suppose that we have a partition $\widetilde{\cP}$ of the full torus $\TND$ into boxes $\widetilde{B}$ of side length $\widetilde{l}_N$, and form a partition $\cP$ refining $\widetilde{\cP}$ by gluing the partitions $(\cP|_{\widetilde{B}}: \widetilde{B}\in \widetilde{\cP})$.  For $p>1$ as above, we find \begin{equation}\begin{split}\label{eq: upscale toy}
			\|(\Lambda_{B}\eta^N)^\alpha\|_{\ell^p(B\in \cP_N)} & \le \|\Delta_{\widetilde{B}}(\eta^N) + 2(\Lambda_{\widetilde{B}}\eta^N)^\alpha\|_{\ell^p(\widetilde{B}\in \widetilde{\cP}_N)} \\ & \le (N/\widetilde{l}_N)^{d(1-1/p)}\|\Delta_{\widetilde{B}}\|_{\ell^1(\widetilde{B}\in \widetilde{\cP}_N)} + 2\|(\Lambda_{\widetilde{B}}\eta^N)\|_{\ell^p(\widetilde{B}\in \widetilde{\cP}_N)} \end{split} 
\end{equation} where the factor $(N/\widetilde{l}_N)^{d(1-1/p)}$ arises from interpolating between $\ell^1$ and $\ell^p$ norms on $\widetilde{\cP}$. Using Assumption \ref{assumptions sketch}a), we may take $p>1$ sufficiently close to 1 so that $N^{1-1/p}\chi_N^{\delta/2}\lesssim 1$, and combining the two previous displays yields an estimate of the form \begin{equation}
	\begin{split} \label{eq: one step toy} &\mathbb{E}_{f_N\Pi^N_a}\left[\|(\Lambda_{B}\eta^N)^\alpha\|_{\ell^p(B\in \cP)}-2\|(\Lambda_{\widetilde{B}}\eta^N)^\alpha\|_{\ell^p(\widetilde{B}\in \widetilde{\cP})} \right] \\ & \hspace{3cm}\le  C\chi_N N^{2-d}\d(f_N) \\&\hspace{3cm} + Cl_N^{-d/2}\left(\frac{\widetilde{l}_N}{l_N}\right)^2\chi_N^{\delta/2}\mathbb{E}_{f_N\Pi^N_a}\left(1+\|(\eta^N)^\alpha\|_{\ell^1(\TND)}\right).  \end{split} 
\end{equation} \step{3. Multiscale Recombination} The final key point is that an estimate proving Theorem \ref{thrm: main integrability estimate} may be obtained by combining estimates of the form \eqref{eq: one step toy}. Using both parts of the Assumption \ref{assumptions sketch}, we may find a finite $K<\infty$ and integer sequences $$l^1_N=1\ll l^2_N \dots \ll l^K_N=N$$ so that each $l^k_N$ divides $l^{k+1}_N$, and so that $(l^{k+1}_N/l_N)^2\chi_N^{\delta/2} \lesssim 1.$ Associated to the $l^k_N$, we construct a sequence of partitions $$ \cP^1=\TND\prec \cP^2\prec \dots \prec \cP^K=\{\TND\} $$ where each element of $\cP^k$ is a box of side length $l^k_N$, and each $\cP^k$ refines $\cP^{k+1}$. We refer to Figure \ref{fig:coarse-fine-box} for an illustration of this scheme. For the finest partition, there is no averaging and $\|(\Lambda_{\cP^1}\eta^N)^\alpha \|_{\ell^p(\cP^1)}=\|(\eta^N)^\alpha\|_{\ell^p(\TND)}$, while at the coarsest scale $(\Lambda_{\TND} \eta^N)^\alpha=\|(\eta^N)^\alpha\|_{\ell^1(\TND)}$. We can therefore obtain a telescoping sum \begin{equation}\begin{split}\label{eq: telescope} &\mathbb{E}_{f_N\Pi^N_a}\|(\eta^N)^\alpha\|_{\ell^p(\TND)}-2^{K-1} \mathbb{E}_{f_N\Pi^N_a} \|(\eta^N)^\alpha\|_{\ell^1(\TND)}\\
&=\mathbb{E}_{f_N\Pi^N_a}\left[\|(\Lambda_{B^1}\eta^N)^\alpha\|_{\ell^p(B^1\in \cP^1)}-2\|(\Lambda_{{B^2}}\eta^N)^\alpha\|_{\ell^p({B^2\in \cP^2})} \right] \\ & + 2\mathbb{E}_{f_N\Pi^N_a}\left[\|(\Lambda_{B^2}\eta^N)^\alpha\|_{\ell^p(B^2\in\cP^2)}-4\|(\Lambda_{{B^3}}\eta^N)^\alpha\|_{\ell^p(B^3\in \cP^3)} \right] \\ &\hs \vdots \\ & + 2^{K-2}\mathbb{E}_{f_N\Pi^N_a}\left[\|(\Lambda_{B^{K-1}}\eta^N)^\alpha\|_{\ell^p(B^{K-1}\in \cP^{K-1})}-2^{K-1}\|(\Lambda_{{B^K}}\eta^N)^\alpha\|_{\ell^p(B^K\in {\cP^K})} \right] \end{split}\end{equation} where each term on the right-hand side is of the same form as the left-hand side of \eqref{eq: one step toy}, and thanks to the choices of $l^k_N$, the prefactor in error term remains bounded. Gathering all the terms, and using statistical stationarity to absorb the error terms into the final term, we get, for a new value of $C$, \begin{equation}
	\label{eq: nearly conclusion toy} \mathbb{E}_{f_N\Pi^N_a} \|(\eta^N)^\alpha\|_{\ell^p(\TND)} \le (2^{K-1}+C)\mathbb{E}_{f_N\Pi^N_a}\|(\eta^N)^\alpha\|_{\ell^1(\TND)}+C\chi_N N^{d-2}\d(f_N). 
\end{equation} We finally impose a mass bound. If we additionally impose that $f^N$ must be supported on $X_{N,b}$, defined in \eqref{eq: def XNB}, a simple interpolation of $\ell^p$-norms shows that the $\ell^1$ norm appearing on the right-hand side may be bounded by an arbitrarily small multiple of the $\ell^p$ norm plus a large $b$-dependent constant, which permits to absorb the first term of the right-hand side into the left-hand side \begin{equation}
	\label{eq: nearly conclusion toy} \mathbb{E}_{f_N\Pi^N_a} \|(\eta^N)^\alpha\|_{\ell^p(\TND)} \le C+C\chi_N N^{d-2}\d(f_N) 
\end{equation} where all the constants are allowed to depend on $b$, but are otherwise uniform in $N, \chi_N, f_N$.  It is standard, using the Feynman-Kac formula (for example, \cite{varadhan1991scaling}, \cite[Proof of Theorem 2.1]{kipnis1998scaling}) that such an estimate is sufficient to prove an estimate of the form claimed in Theorem \ref{thrm: main integrability estimate}, completing the proof.  	\\ \\ We conclude the sketch proof by by discussing the role of Assumptions \ref{assumptions sketch}. Hypothesis b), that $N$ is a dyadic integer, is only used in step 3 to ensure the existence of the scales $l^k_N$ such that each $l^{k+1}_N/l^k_N$ is an integer, so that all the boxes may be taken of the same size. In the main text, the issue is resolved for all $N$ by introducing {\em weightings}: We omit this from the sketch proof, because the resolution introduces no further essential difficulties in the analysis, but makes the notation more cumbersome. In contrast, hypothesis a) is crucially used in two places in the argument sketched above: first to choose $p>1$ to control the error terms, and secondly in ensuring that the iterative scheme terminates after a finite (possibly large, but $N$-independent) number of steps. The first issue is solved by tailoring a strict Young function $\Phi$ and working with the resulting Orlicz norm rather than $\ell^p$-norms. For the second issue, generalising to $K=K_N\to\infty$ requires introducing a further tuneable parameter $\lambda\in (0,1]$, which keeps the prefactors appearing in the analogue of telescoping sum \eqref{eq: telescope} bounded.
		\subsection{Organisation of the Paper} The paper is structured as follows. The remainder of this section is a literature review and discussion. Section \ref{sec: prelim} recalls some definitions of frequent use. The key technical content is in Section \ref{sec: int}, in which the above sketch proof of Theorem \ref{thrm: main integrability estimate} is made rigorous: Along the way, machinery helpful for the superexponential estimate is established, and this section is therefore presented first. Section \ref{sec: ET} proves Proposition \ref{prop: tightness}, and finally Section \ref{sec: supex} gives the proof of Theorem \ref{thrm: supex}. Appendix \ref{sec: Orlicz Appendix} provides some technical tools on Orlicz norms, deferred from Section \ref{sec: int}.

		\subsection{Literature Review and Discussion} \subsubsection{Particle Systems with Unbounded and Degenerate Diffusion} We first review some of the literature on particle models that show at least one of the two features of unbounded and degenerate diffusion of interest. 
		A `stick-breaking' process, which produces the PME in the limit, has been studied in the hydrodynamic limit  \cite{suzuki1993hydrodynamic,ekhaus1995stochastic, feng1997microscopic}. Notably, in these works, the integrability estimate is carried out only in expectation and the argument cannot be generalised to the LDP analysis. Recently, a sequence of works  have derived both porous medium \cite{goncalves2009hydrodynamic,blondel2018convergence,goncalves2023exclusion} and fast diffusion equations \cite{goncalves2023exclusion} equations from exclusion processes with kinetic constraints.  The rescaling of the zero-range process considered in the present work was previously introduced in \cite{gess2023rescaled}, and a hydrodynamic limit and large deviation principle were shown under a condition that $\chi_N\to 0$ sufficiently fast. The rate function is the same as that derived here, and a geometric meaning as an energy-dissipation inequality was derived \cite[Theorem 4]{gess2023rescaled}. We also refer to \cite{landim1995large,landim1997hydrodynamic,kipnis1998scaling,bernardin2022non,menegaki2021quantitative,menegaki2022consistence} for further works on the hydrodynamic limit and large deviations of the zero-range process.  		\subsubsection{Multiscale Methods} The use of different scales for obtaining estimates in LDP is classical and goes back to the one- and two- block estimates in the Varadhan method \cite{varadhan1991scaling,kipnis1989hydrodynamics,kipnis1990large}. The appearance of numerical-type continuity from the regularity of laws is tacitly part of the standard proof of the two-block estimate, see also \cite[Proposition 3.5]{fgh2025} for a discussion of the connection. Numerical-type regularity, and a multi-scale approach to a replacement lemma is obtained via the Dirichlet form in \cite[Lemma 4.4, Theorem 5.1]{quastel1999large}. Energy estimates, obtaining the large-scale regularity, are also obtained from the Dirichlet form in \cite{bertini2009dynamical,kipnis1998scaling}. To the best of the authors' knowledge, the current work represents the first time that the regularity across scales in LDP has been systematised, both in terms of i) expressing mesoscopic scale regularity through estimating the regularity through an intrinsic regularity \eqref{eq: superexp ent diss}, \eqref{eq: preintro mesoscopic est} dictated by the large-scale regularity of the skeleton equation \cite{fehrman2019large}, and ii) sharply identifying the scales across which regularity can be estimated.\medskip \\ The approach we develop here was inspired by other works in the study of lattice systems, most particularly the works of Bauerschmidt-Bodineau \cite{bauerschmidt2020spectral,bauerschmidt2021log} and Bauerschmidt-Park-Rodriguez \cite{bauerschmidt2024discrete} exploiting averaging on a `hierarchical' decomposition of the lattice. The proof technique in Section \ref{sec: int} also somewhat resembles the coarse-graining approach to renormalisation group analysis \cite{wilson1971renormalization}: Lemma \ref{lemma: reg by paths} may be understood as establishing a relationship between the Dirichlet form of the true process and that of an effective process at scale $l_N$, where particles may jump from any site in a block to any site in a neighbouring block, and the total rate of such jumps depends only on the $\alpha$-averages on the entire blocks. Let us also refer to \cite{frohlich1981kosterlitz,frohlich1983berezinskii} for an example multiscale summation in the renormalisation group.  The use of regularity appearing across multiple scales has also been exploited in the theory of stochastic homogenisation \cite{gloria2017quantitative,armstrong2019quantitative,gloria2020regularity,gloria2021quantitative}. The philosophy resembles that of \cite{armstrong2016quantitative,armstrong2021homogenization}, where large-scale regularity is used to aid analysis on smaller scales, but the proof technique is different: In the proof of Theorem \ref{thrm: main integrability estimate}, a single mechanism is responsible for the creation of (suitably measured) regularity at all scales.   \subsubsection{LDP of SPDEs and Macroscopic Fluctuation Theory} The study of large deviations of microscopic systems is closely related to macroscopic fluctuation theory (MFT) \cite{bertini2006non,bertini2015macroscopic,derrida2007non}, which describes the non-equilibrium statistical mechanics and fluctuations of observables via an Ansatz for the large-deviations rate function \cite{bertini2015macroscopic}. This has motivated the study of {\em fluctuating hydrodynamics} \cite{dean1996langevin,donev2014reversible,donev2014dynamic} and SPDEs with conservative noise, constructed so that the large deviation rate functions coincide with those of a given particle system. The same rate function $\mathcal{I}_\rho$ in Theorem \ref{thrm: LDP} and derived in \cite{gess2023rescaled} corresponds to the conservative SPDE \eqref{eq: SPDE}, see also \cite{dirr2016entropic,giacomin1999deterministic}. Although SPDEs in this form generally admit at most only trivial solutions without the mollification \cite{konarovskyi2019dean}, the well-posedness of SPDEs with a conservative noise term, including \eqref{eq: SPDE}, was established by  \cite{fehrman2021well} as soon as $\delta>0$. We refer to \cite{cerrai2011approximation,faris1982large,jona-lasinio1990large,hairer2015large} for works on the large deviations of SPDEs; in general, scaling relations between noise intensity and correlation may lead to renormalisation counterterms appearing in the rate function \cite{hairer2015large}. A large deviation principle for \eqref{eq: SPDE} is proven in \cite[Theorem 6.8]{fehrman2019large}, for a wider class of nonlinearities $\varphi(u)$ replacing the porous medium nonlinearity $u^\alpha$, and \cite{dirr2020conservative} proved the analogous large deviation principle corresponding to the simple symmetric exclusion process, in both cases under a scaling relation on $\epsilon\delta^{-(d+2)}\to 0$ sufficient to guarantee large-scale pathwise regularity.  \subsubsection{Parameter Relations in LDP} We first comment on Theorem \ref{thrm: LDP} in light of the analysis in \cite{gess2023rescaled}. The previous analysis and a first attempt at the one-block estimate led to the conjectural dichotomy between the regimes $\chi_NN^2\lesssim 1$ and $\chi_N N^2\to \infty$, see \cite[Figure 1]{gess2023rescaled}, depending on whether pathwise regularity or local equilibration dominates. The argument presented in Section \ref{sec: sketch} unveils a deeper structure. If $\chi_N N^2\lesssim 1$, the multiscale argument requires only one step; if $\chi_N N^2\to \infty$ but $\chi_N N^{2/2}\ll 1$, then two steps of the multiscale estimate are required, and in general, if $\chi_N N^{2/(k-1)}\to \infty$ but $\chi_N N^{2/k}\ll 1$ then the multiscale iteration sketched above terminates after $k$ steps of rescaling.   \\ There is a natural correspondence between the parameters of the rescaled ZRP considered here and the parameters of SPDEs such as \eqref{eq: SPDE}. Comparing the scale fluctuations at each point and the correlation lengths leads to the identifications $ \epsilon \delta^{-d} \sim {\chi_N},  \delta \sim N^{-1} $. In this way, the speeds of the large deviation rate functions $\epsilon^{-1}\sim N^d/\chi_N$ coincide, and the scaling relation $\chi_N N^2\lesssim 1$ in \cite{gess2023rescaled} matches to the condition $\epsilon \delta^{-d-2}\lesssim 1$ needed for large-scale spatial regularity in the SPDE setting \cite{fehrman2019large,dirr2020conservative}. By contrast, the result of the current paper corresponds to any scaling regime $\epsilon \delta^{-d}\to 0$, suggesting that the same methods may allow to derive large deviation principles under this weaker condition for SDPEs.\\
        Notably, in the SPDE setting, formal power counting arguments, in the spirit of \cite{hairer2018class}, show that \eqref{eq: SPDE} is subcritical exactly in the joint scaling regime $\epsilon \delta^{-d}\to 0$. This might indicate that significant challenges in the method developed here would have to be expected {in the case $\chi_N\sim 1$, which we are currently unable to handle}.
%        We are unable to handle the case $\chi_N=\chi>0$, corresponding to the scaling-critical case $\epsilon \delta^{-d}\sim 1$, where we anticipate the appearance of an additional term arising from the It\^o-Stratonovich correction, similar to the appearance of a massive correction in the rate function in \cite{hairer2015large}.
        \section{Preliminaries} 
 \label{sec: prelim} \subsection{The Process} We first give a precise definition of the process $\eta^N_\bullet$ of study, see also \cite{gess2023rescaled}. We consider the zero-range process on a discrete torus $\{0, 1, \dots, N-1\}^d$ with a jump rate $g(k)=dk^\alpha$. In a configuration $k$ where there are $k(x) 
		\in \mathbb{N}, x\in \{0, 1, \dots, N-1\}^d$ particles at each site, an exponential clock of rate $g(k(x))$ is initialised at each site; when the first clock rings, a particle is moved from site $x$ to one of its neighbours $y\sim x$, chosen randomly and independently of the clocks, and the process begins again. This defines a Markovian dynamics in $k\in \mathbb{N}^{\{0,1,\dots,N-1\}^d}$. \\ It is standard to perform a hydrodynamic rescaling of space and time in order to obtain a hydrodynamic limit. As in \cite{gess2023rescaled}, a further rescaling is necessary to produce degenerate diffusion, introducing the parameter $\chi_N\to 0$. We set $$ \eta^N_t(x):= \chi_N k\left(N^2 t\chi_N^{\alpha-1}, N x\right) $$ which amounts to, in addition to the usual hydrodynamic rescaling, counting of each particle as of size $\chi_N$, and making the local jump rate $2d(\eta^N(x))^\alpha/\chi_N$. The rescaled spatial variable $x$ will now run over the rescaled discrete torus $$\TND:=\left\{0, N^{-1}, \dots, \frac{N-1}{N}\right\}^d \subset [0,1)^d$$  and $\eta^N$ defines a Markov process in the state space $X_N:=(\chi_N \mathbb{N})^{\TND}$ with generator given by $N^2 \cL_N$, where \begin{equation}
	\label{eq: linear generator} \cL_N F(\eta^N):=\frac12\sum_{x\sim y} \frac{(\eta^N(x))^\alpha}{\chi_N}\left(F(\eta^{N,x,y})-F(\eta^N)\right).
\end{equation}
In this definition, the sum runs over all pairs $\{x, y\}$ of neighbouring sites in $\TND$, and $\eta^{N,x,y} := \eta^N+\chi_N 1_y-\chi_N 1_x$ is the configuration obtained from $\eta^N$ by moving one particle from $x$ to $y$. Here, and in the sequel, we write $1_x$ for the function on $\TND$ which is equal to 1 at $x\in \TND$, and zero elsewhere. It will be convenient to isolate the invariant sets $X_{N,b}$ of the subset of the state space given by imposing an upper bound on the mass $$ X_{N,b}:=\left\{\eta^N\in X_N: \frac1{N^d} \sum_{x\in \TND} \eta^N(x)\le b\right\}. $$  
A standard calculation shows that the process admits a continuum family of reversible equilibrium distributions\begin{equation}\label{eq: eq}
	\Pi^N_a(\eta^N):=\prod_{x\in \TND} \frac{(a/\chi_N)^{\alpha(\eta/\chi_N)}}{((\eta/\chi_N)!)^\alpha Z_\alpha(a/\chi_N)}, \qquad  \rho>0
\end{equation} where $Z_\alpha$ is a suitable normalisation constant. Where $\rho\in C(\TTd,(0,\infty)$ is a profile varying at the macroscopic scale, we define the slowly varying local equilibrium with the same notation: \begin{equation}\label{eq: svleq}
	\Pi^N_\rho(\eta^N):=\prod_{x\in \TND} \frac{(\rho(x)/\chi_N)^{\alpha(\eta/\chi_N)}}{((\eta/\chi_N)!)^\alpha Z_\alpha(\rho(x)/\chi_N)}, \qquad  \rho>0
\end{equation} It may easily be verified that the equilibrium distributions have the {\em integration by parts} property that, for any $F: X_N\to \mathbb{R}$ and any $x\in \TND$, \begin{equation} \label{eq: int by parts}
	\mathbb{E}_{\Pi^N_\rho}[F(\eta^N){(\eta^N(x))^\alpha}]=\rho^\alpha \EE_{\Pi^N_\rho}[F(\eta^N+\chi 1_x)].
\end{equation} \paragraph{\em{Conventions on Lattice Geometry}} We have chosen the definition of the process $\eta^N_\bullet$ to already include the rescaling of space and time, and as a result, the particle sites $x, y$ will always take values in the rescaled torus $\TND$. Correspondingly, when speaking of the `side length' of a box $B\subset \TND$, we will always mean the microscopic side length, namely the number of edges $l \ge 1$ of the discretisation $\TND$ on each edge of the box. Since we will always use `side length' to mean microscopic side length, we will not repeat the `microscopic' in the sequel.\\ \\  We will refer to the {\em lattice distance} of two points $x, y$ to be the length $k\ge 1$ of the shortest path $x_0=x, x_1, \dots, x_k=y$ with $x_j\in \TND$ and each pair $x_j, x_{j+1}$ adjacent in $\TND$. In contrast, $|\cdot|_\infty$ will always be used to denote the macroscopic distance between two points on the torus $\TTd$ induced by the $\ell^\infty$-norm on the Euclidean space $\RRd$.
\subsection{Feynman-Kac formula and Dirichlet Form} The main focus of the paper will be in obtaining the large deviation estimates, encoded in Theorems \ref{thrm: supex} and \ref{thrm: main integrability estimate}, necessary to prove Theorem \ref{thrm: LDP}. The basic tool we will use is the Feynman-Kac formula, see \cite{kipnis1989hydrodynamics}, \cite[Section 10.3]{kipnis1998scaling}, which reduces exponential estimates to a variational problem. In the context of the model introduced above, with the large deviation speed $N^d/\chi_N$, this formula shows that, for any $N$ and any function $F: X_N\to \mathbb{R}$, \begin{equation}
	\label{eq: FK formula} \begin{split} & \frac{\chi_N}{N^d}\log \mathbb{E}_{\Pi^N_a}\left[\exp\left(\frac{N^d}{\chi_N}\int_0^\tf F(\eta^N_s)ds \right)\right] \\& \hspace{2cm} \le \sup_{f_N}\left\{\mathbb{E}_{f_N\Pi^N_a}[F(\eta^N)]-\chi_N N^{2-d}\d_N(f_N)\right\}. \end{split}
\end{equation} Here, the supremum on the right-hand side runs over all probability density functions $f_N$ on $X_N$ with respect to the invariant measure $\Pi^N_a$, and $\d_N$ is the Dirichlet form corresponding to the generator \eqref{eq: linear generator}, displayed at \eqref{eq: dirichlet on laws}. Thanks to the Feynman-Kac formula, the proofs of Theorems \ref{thrm: supex}-\ref{thrm: main integrability estimate} reduce to estimating suitable expectations in terms of Dirichlet forms, and we henceforth make this reduction without further comment. \subsection{Topology of the Large Deviations Principle} In order to give meaning to a large deviations principle, we must specify a suitable topological space. It is straightforward to associate each configuration $\eta^N\in X_N$ to a nonnegative function on $\TTd$, for example extending $\eta^N$ to be constant on small cubes around each lattice point. With this identification, we let $d$ be a metric on $L^1(\TTd, [0,\infty))$ corresponding to the weak convergence $\langle \varphi, u_n\rangle \to \langle \varphi, u\rangle, \varphi \in C(\TTd)$, and take the path space $\mathbb{D}$ to be the  Skorokhod space  \begin{equation}
	\mathbb{D}:=\mathbb{D}\left([0,\tf], (L^1(\TTd,[0,\infty)), d)\right). 
\end{equation} It is relative to this topology which we will prove exponential tightness in Proposition \ref{prop: tightness}. 

\section{Large Deviation Integrability Estimate} \label{sec: int} The aim of this section will be to prove Theorem \ref{thrm: main integrability estimate}. We prove the following, which is sufficient thanks to the Feynman-Kac formula \eqref{eq: FK formula}. \begin{proposition}[Integrability via Dirichlet Form] \label{prop: int dirichlet form} Let $\chi_N\to 0$ be any sequence, and fix $a, b>0$. Then there exist a strict Young function $\Phi$ and $\vartheta=\vartheta(a,b)>0$ such that \begin{equation}
					\label{eq: main integrability conclusion'} \limsup_N\sup_{f_N}\left\{\mathbb{E}_{f_N\Pi^N_a}\left[\|(\eta^N)^\alpha\|_{\Phi}1(\eta^N\in X_{N,b})\right]-\vartheta \chi_N N^{2-d}\mathfrak{d}_N(f_N)\right\}<\infty
				\end{equation} where the supremum runs over all probability density functions $f_N$ with respect to $\Pi^N_a$. \end{proposition} The section is structured as follows. In Section \ref{sec: int prelim}, we give precise definitions of the various objects needed to make the sketch proof of Theorem \ref{thrm: main integrability estimate} in the introduction rigorous. Some preliminary stochastic estimates are in Section \ref{sec: int basic}, and in Section \ref{sec: int main} we give the proof of Proposition \ref{prop: int dirichlet form}.
		
\subsection{Preliminaries}\label{sec: int prelim} We first set up the technical machinery needed to make the sketch proof in Section \ref{sec: sketch} rigorous.  \paragraph{\emph{Box Partitions}} Given a scale $1\le l_N\le N$, let us write $\mathcal{Q}(N, l_N)$ for the set of partitions $\mathcal{P}_N$ of the torus $\mathbb{T}_N^d$ of the form \begin{equation}\label{eq: prototype P}\mathcal{P}=\left\{B=\prod_{i=1}^n A_i: A_i\in \mathcal{P}_N^i\right\}\end{equation} where each $\mathcal{P}_N^i$ is a partition of the 1-dimensional torus into intervals of length between $l_N$ and $2l_N$.  Each such $\mathcal{P}$ comes with a natural notion of \emph{adjacency}, where we say that two blocks $B, B'\in \mathcal{P}$ are adjacent, and write $B\sim_{\mathcal{P}} B'$, if they can be expressed as $B=\prod_i A_i, B'=\prod_i A'_i$, with $A_i, A'_i\in \mathcal{P}^i_N$, such that $A_i=A'_i$ for all but one $i$, and on this index, there exist adjacent $x\in A_i, y\in A'_i$, so that $A_i, A'_i$ are adjacent elements of $\mathcal{P}^i_N$.  \\\\  Given $\mathcal{P}_N\in \mathcal{Q}(N,l_N)$ and $\widetilde{\mathcal{P}}_N\in \mathcal{Q}(N, \widetilde{l}_N)$, we write $\mathcal{P}_N\prec \widetilde{\mathcal{P}}_N$ to denote that $\mathcal{P}_N$ is a refinement of $\widetilde{\mathcal{P}}_N$, that is, that every $B\in \mathcal{P}_N$ is contained in a (unique) $\widetilde{B}\in \widetilde{\mathcal{P}}_N$. We write $\pi: \cP_N\to \widetilde{\cP}_N, \pi(B):=\widetilde{B}$ for the map taking each block to the element of the coarser partition which contains it. \\ \paragraph{\emph{Weightings}} We introduce weightings on the underlying set $\TND$, which are then inherited by all partitions. In the sequel, a weighting will be a strictly positive measure on the discrete $\sigma$-algebra of $\TND$, which we identify with a function $w^N: \TND\to (0,\infty)$. \\     \paragraph{\emph{Orlicz Norms on Block Functions}} Fix a partition $\mathcal{P}_N$ of $\mathbb{T}^d_N$, a weighting $w^N$. We recall that a {\em Young function} is a map $\Phi:\mathbb{R}\to [0,\infty]$ which is convex, even, lower semicontinuous and neither the zero function nor its convex dual. We will always work with Young functions which are additionally {\em strict}, meaning that \begin{equation}
	\label{eq: strict}\Phi(u)<\infty\text{ for all }u\in \mathbb{R}; \qquad \frac{\Phi(u)}{u}\to \infty \text{ as }u\to \infty
\end{equation} and to ensure the finiteness of all norms, $\Phi(0)=0$. In the sequel the uppercase $\Psi$ will be reserved for the convex dual of a Young function, given by \begin{equation}\label{eq: convex dual}
	\Psi(x):=\sup\{ux-\Phi(u): u\in \mathbb{R}\}.
\end{equation} The definition is such that $\Phi$ is strict if, and only if, $\Psi$ is, and $\Phi$ may be recovered from $\Psi$ by the same formula \eqref{eq: convex dual}. We define the Orlicz norm on functions $h: \mathcal{P}_N\to \mathbb{R}$ given by \begin{equation} \label{eq: orlicz} \|h(B)\|_{l_\Phi(B\in\mathcal{P}_N, w^N)}:=\inf\left\{t\ge 0: \sum_{B\in \mathcal{P}_N} \frac{w^N(B)}{w^N(\TND)}\Phi\left(\frac{h(B)}{t}\right)\le 1 \right\}.\end{equation} Similarly, when $\widetilde{B}\subset \TND$ is a union of sets from $\cP_N$, we write $\cP_N|_{\widetilde{B}}$ for the resulting partition of $\widetilde{B}$, and for a function $h:\cP_N|_{\widetilde{B}}\to \mathbb{R}$, we define \begin{equation}\label{eq: orlicz'} \|h(B)\|_{l_\Phi(B\in\cP_N|_{\widetilde{B}},w^N)}:= \inf\left\{t\ge 0: \sum_{B\in \mathcal{P}_N} \frac{w^N(B)}{w^N(\widetilde{B})}\Phi\left(\frac{h(B)}{t}\right)\le 1 \right\}. \end{equation} We similarly write $\|\cdot\|_{\ell^p(B\in \mathcal{P}_N, w^N)}, \|\cdot\|_{\ell^p(B\in {\mathcal{P}}_N|_{\widetilde{B}}, w^N)}$ for the Lebesgue norms defined with the same normalised measure. {Here, and throughout, we use the nonstandard notation of writing the function with an argument $h(B)$ inside the norm, and identifying the argument $B\in \cP_N$ in the subscript, in order to improve clarity of certain calculations. }Finally, in the case of the trivial partition $\cP_N=\TND$ and $w^N=1$, we omit both arguments and write $\|\cdot\|_{\ell_\Phi}, \|\cdot\|_{\ell^p}$. \\ \paragraph{\emph{Integral Lemmata}} We now present some facts about these norms which we will use in the proof.  We will make use of the following discrete Sobolev-Gagliardo-Nirenberg embedding. \begin{proposition}\label{prop: discrete Sobolev} Let $B$ be a lattice box of dimension $d$, and side lengths between $l, al$, for some $a>1$, equipped with the constant weight $w=1$. For functions $f: B\to \mathbb{R}$, define \begin{equation}
			\|\nabla f(x)\|_{\ell^2(x\in B,1)}^2:=l^{2-d}\sum_{x,y\in X: x\sim y}|f(x)-f(y)|^2.
		\end{equation} For the uniform weighting $w=1$, and for some $p=p(d)>1$, for all $f:B\to \mathbb{R}$ and all $\lambda>0$, it holds that \begin{equation}
			\|f^2(x)\|_{\ell^p(x\in B,1)}\le (1+\lambda)\|f(x)\|^2_{\ell^2(x\in B,1)}+C(1+\lambda^{-1})\|\nabla f(x)\|_{\ell^2(x\in B,1)}^2.
		\end{equation} \end{proposition} This postprocess the usual (discrete) Sobolev-Gagliardo-Nirenberg embedding \cite[Section 4]{porretta2020note} with the Poincar\'e inequality \cite[Equation (4.6)]{diaconis1996logarithmic}, and using a Peter-Paul argument to make the prefactor in from of the $\ell^2$-norm arbitrarily close to $1$.  \\\\ The next result, whose proof we defer to Appendix \ref{sec: Orlicz Appendix}, will allow us to iterate averaging procedures. \begin{lemma}\label{lemma: orlicz consistency} Let $\cP_N\prec \widetilde{\cP}_N$ be partitions of $\TND$, let $h:\cP_N\to \mathbb{R}$, and assume that $p>1$ and a strictly increasing Young function $\Phi$ are such that the composition\begin{equation}\label{eq: Young-p consistency} \Theta:\mathbb{R}\to \mathbb{R}, \quad u\mapsto (\Phi^{-1}(u))^p \quad \text{ is convex}.\end{equation} For a weighting $w^N$, let $\widetilde{h}:\widetilde{\cP}_N\to \mathbb{R}$ be given by $$ \widetilde{h}(\widetilde{B}):=\|h(B)\|_{\ell^p(B\in \cP_N|_{\widetilde{B}}, w^N)}.$$ Then it holds that \begin{equation}
	\label{eq: orlicz consistency} \|h(B)\|_{\ell_\Phi(B\in\cP_N, w^N)}\le \|\widetilde{h}(\widetilde{B})\|_{\ell_\Phi(\widetilde{B}\in\widetilde{\cP}_N,w^N)}.
\end{equation}
	
\end{lemma} The proof is given in Appendix \ref{subsec: orlicz consistency}.  The final lemma regarding these norms will allow us to interpolate between the nonlinear functionals $\|u^\alpha\|_{\ell_\Phi}, \|u^\alpha\|_{\ell^1}$, given only control of a lower order, linear norm $\|u\|_{\ell^1}$. \begin{lemma} \label{lemma: orlicz interpolation}
	Set $\cP_N=\TND$ and $w^N=1$, and let $\Phi$ be any strict Young function. Then for any $\delta>0$ and $b\ge 0$, there exists $z=z(\Phi, b, \delta)<\infty$ such that, for all $u:\TND\to [0,\infty)$ with $\|u\|_{\ell^1}\le b$, \begin{equation}
	\|u^\alpha\|_{\ell^1}\le \delta \|u^\alpha\|_{\ell_\Phi}+z.
	\end{equation}
\end{lemma} 
 The proof of this lemma is given in Appendix \ref{subsec: orlicz interpolation}.

\paragraph{\emph{$\alpha$-averaging}} Finally, we define an averaging operation on configurations $\eta^N\in X_N$ which interacts well with the nonlinearity $u\mapsto u^\alpha$. For any weighting $w^N$ and any subset $\emptyset \not = B \subset \TND$ and configuration $\eta^N\in X_N$, we will write $\Lambda_{B, w_N} \eta^N$ for the ``$\alpha$-averaging" \begin{equation}
	\label{eq: alpha avg}\Lambda_{B, w^N} \eta^N:=\left(\frac{1}{w_N(B)}\sum_{x\in B}w^N(x)(\eta^N(x))^\alpha \right)^{1/\alpha}.
\end{equation} \paragraph{\emph{Notation}} When the weighting $w^N$ is fixed, it will be omitted from the notation for the $\alpha$-averaging in order to ease notation. 
\subsection{Stochastic Estimates}\label{sec: int basic} We now turn to the problem of estimating various expectations $\mathbb{E}_{f_N \Pi^N_a}[F]$ in terms of $\chi_N N^{2-d}\d_N(f_N)$, providing the basic ingredients corresponding to the steps in the outline proof in Section \ref{sec: sketch}.  \\ The first point is that we may freely restrict to shift-invariant density functions $f_N$. We define the set of Galilean-invariant functions $G_N(X_N)$ to be those functions $F:X_N\to \mathbb{R}$ such that $F\circ \tau_x=F$ for any translation $\tau_x$, and whenever $R: X_N\to X_N$ is a composition of rotations and reflections preserving the lattice, $F=F\circ R$. Similarly, given an invariant measure $\Pi^N_a$, we let $G_N(\Pi^N_a)$ denote those probability density functions $f_N$ with respect to $\Pi^N_a$ which are invariant under the same transformations. Finally, for any $b>0$, we denote $G_N(\Pi^N_a,b)$ those probability densities $f_N\in G_N(\Pi^N_a)$ which are supported on $X_{N,b}$. With this notation, the principle we use is as follows. \begin{lemma}\label{lemma: gallilean pdfs}[Restriction to Shift-Invariant Measures] For any $N, a>0, \lambda>0$ and $F\in G_N(X_N)$, it holds that \begin{equation}
	\sup_{f_N}\left\{\mathbb{E}_{f_N\Pi^N_a}[F]-\lambda \mathfrak{d}_N(f_N)\right\}=\sup_{f_N \in G_N(\Pi^N_a)}\left\{\mathbb{E}_{f_N\Pi^N_a}[F]-\lambda \mathfrak{d}_N(f_N)\right\}. 
\end{equation} Moreover, if, for some $b$, $\{F>0\}\subset X_{N,b}$, then the supremum on the right-hand side can be further restricted to $f_N\in G_N(\Pi^N_a,b)$.
	
\end{lemma}
Since this is a well-known principle consisting entirely of `soft' (qualitative) arguments, we omit the proof and refer to \cite[Proof of Lemma 4.1, Step 3]{kipnis1998scaling}.
%\begin{proof}[Sketch proof]
%	Assume first that $F\in G_N(X_N)$, and let $f_N$ be any competitor for the left-hand side. Setting $\bar{f}_N:=\frac{1}{2^d d!N^d}\sum_R\sum_x f_N\circ \tau_x\circ R$ produces $\bar{f}_N \in G_N(\Pi^N_a)$, the invariance of $F$ implies that the expectation of $F$ remains the same, and the convexity of the Dirichlet form $\d_N$ implies that the penalty term will decrease. This implies that the maximum is obtained on $G(\Pi^N_a)$ as claimed. If $F$ is supported on $X_{N,b}$, then an arbitrary competitor $f_N \in G_N(\Pi^N_a)$ may be decomposed as $f_N=pf^1_N+(1-p)f^2_N$, with $p\in [0,1]$, $f^1_N \in G_N(\Pi^N_a, b)$, and $f^2_N\in G_N(\Pi^N_a)$ supported on $X^N\setminus X_{N,b}$. Since $X_{N,b}, X\setminus X_{N,b}$ are invariant for the dynamics, the Dirichlet form decomposes as $$ \d_N(f_N)=p\d_N(f^1_N)+(1-p)\d_N(f^2_N)$$ whence \begin{equation*}\begin{split} \mathbb{E}_{f_N\Pi^N_a}[F]-\lambda \mathfrak{d}_N(f_N) &= p(\mathbb{E}_{f^1_N\Pi^N_a}[F]-\lambda \mathfrak{d}_N(f^1_N))-\lambda(1-p)\d_N(f^2_N) \\ & \le \mathbb{E}_{f^1_N\Pi^N_a}[F]-\lambda \mathfrak{d}_N(f^1_N).\end{split}\end{equation*}\end{proof}
The next step produces, for shift-invariant densities $f_N$, a canonical-paths type estimate.
\begin{lemma}\label{lemma: really useful observation} For any $f_N\in G_N(\Pi^N_a)$ and any lattice sites $x,y \in \TND$ at lattice distance $k\in [1, Nd]$, it holds that \begin{equation}
	\label{eq: ruo} \int_{X_N}\frac{(\eta^N(x))^\alpha}{\chi_N}\left[\sqrt{f(\eta^{N,x,y})}-\sqrt{f(\eta^N)}\right]^2 \Pi^N_a(d\eta^N) \le k^2N^{-d}\mathfrak{d}_N(f_N). 
\end{equation}\end{lemma} %This may be deduced from \eqref{eq: int by parts} in exactly the same way as, for example, \cite[Chapter 5, Section 5.5]{kipnis1998scaling}. 
\begin{proof}Thanks to the integration by parts property \eqref{eq: int by parts}, the left-hand side is equal to $$ \chi_N^{-1} a^\alpha \int_{X_N} \left[\sqrt{f(\eta^N+\chi_N 1_y)}-\sqrt{f(\eta^N+\chi_N 1_x)}\right]^2 \Pi^N_a(d\eta^N) $$ where, as in the definition of the process, $\eta^N+\chi_N 1_{x} \in X_N$ denotes the configuration with one particle added at the side $x\in \TND$. We now take a path $x_0=x, x_1, \dots, x_{k-1}, x_k=y, x_j\in\TND$ whose length is the lattice distance $k$, and write the difference as the telescoping sum \begin{equation}	\sqrt{f(\eta^N+\chi_N 1_y)}-\sqrt{f(\eta^N+\chi_N 1_x)}=\sum_{j=1}^k \sqrt{f(\eta^N+\chi_N 1_{x_j})}-\sqrt{f(\eta^N+\chi_N 1_{x_{j-1}})}.\end{equation} Squaring and using Cauchy-Schwarz, the integral is at most \begin{equation}
	k \chi_N^{-1} a^\alpha \sum_{j=1}^k\int_{X_N}\left[\sqrt{f(\eta^N+\chi_N 1_{x_j})}-\sqrt{f(\eta^N+\chi_N 1_{x_{j-1}})}\right]^2 \Pi^N_a(d\eta^N).
\end{equation} Since $f_N \in G(\Pi^N_a)$ by assumption and each $x_{j-1}, x_j$ is a bond of the lattice, each term of the sum produces the same contribution, and reversing the integration by parts produces an upper-bound for the left-hand side of $$ k^2  \int_{X_N} \frac{(\eta^N(x))^\alpha}{\chi_N}\left[\sqrt{f(\eta^{N, x, x_1})}-\sqrt{f(\eta^N)}\right]^2\Pi^N_a(d\eta^N).$$ The integrand is now the contribution to $\d_N(f_N)$ from the bond $x, x_1$, and since $f_N$ is translationally invariant, each of the $dN^d$ bonds of the lattice produces the same contribution. The remaining integral may therefore be estimated by $N^{-d}\d_N(f_N)$, which yields the claimed bound.   \end{proof}

We are finally ready to use Lemma \ref{lemma: really useful observation} to make a statement about the averages of $\eta^N$ on nearby lattice sites, or $\alpha$-averages of $\eta^N$ on neighbouring lattice boxes. This makes good the claim in the introduction that we can obtain pathwise regularity across suitably chosen mesoscopic scales from regularity in probability space at the level of the Dirichlet form, and will be at the heart of the proof of Theorem \ref{thrm: main integrability estimate}, c.f. \eqref{eq: preintro mesoscopic est} in the outline proof, and will additionally be used in the proof of Theorem \ref{thrm: supex} in Section \ref{sec: supex}. We recall the notation $\Lambda$ from \eqref{eq: alpha avg} for the $\alpha$-averaging operation.\begin{lemma} \label{lemma: reg by paths}[Pathwise Regularity from Regularity of Laws]
	There exists a constant $C$ such that, for all $N$, for all $f_N\in G_N(\Pi^N_a)$ and lattice points $x,y$ at lattice distance $k$, \begin{equation} \begin{split} \label{eq: reg by path conc}
		\mathbb{E}_{f_N\Pi^N_a}\left[\left((\eta^N(x))^{\alpha/2}-(\eta^N(y))^{\alpha/2}\right)^2\right]  & \le \chi_N k^2N^{-d}\mathfrak{d}_N(f_N) \\ &+C\chi_N^\delta\mathbb{E}_{f_N\Pi^N_a}(1+\|(\eta^N)^\alpha\|_{\ell^1(\TND)})\end{split}
	\end{equation} where the exponent $\delta$ is given by $ \delta:=\min(1, \frac\alpha2).$ Moreover, if $B, B'$ are lattice boxes of side length $l_N$, such that any $x\in B, y\in B'$ can be joined by a path of length $\widetilde{l}_N$, and $w^N$ is a weighting satisfying $\frac{\sup w^N}{\inf w^N}\le A$, then  \begin{equation} \begin{split} \label{eq: reg by path conc'}
		\mathbb{E}_{f_N\Pi^N_a}\left[\left((\Lambda_{B}\eta^N)^{\alpha/2}-(\Lambda_{B'}\eta^N)^{\alpha/2}\right)^2\right]  & \le \chi_N \widetilde{l}_N^2N^{-d}\mathfrak{d}_N(f_N) \\ &+CA^{1/2}l_N^{-d/2}\chi_N^\delta\mathbb{E}(1+\|(\eta^N)^\alpha\|_{\ell^1(\TND)}).\end{split}
	\end{equation}

\end{lemma}
\begin{proof}
	First, we observe that for any $\theta\in \mathbb{R}$, points $x,y$ at a lattice distance $k$ and $\eta^N\in X_N$, it holds that \begin{equation} \label{eq: basic inequality for dirichlet}
		-f_N(\eta^N)(e^\theta-1)-f_N(\eta^{N,x,y})(e^{-\theta}-1)\le \left(\sqrt{f_{N}(\eta^N)}-\sqrt{f_{N}(\eta^{N,x,y})}\right)^2.
	\end{equation} For the first claim (\ref{eq: reg by path conc}), we take $\theta=\theta(\eta^N, \eta^{N,x,y})$ in the previous inequality to be \begin{equation}
		\theta:=\frac{\alpha}{2}\left(\log (\eta^N(y)+\chi_N)-\log \eta^N(x)\right)
	\end{equation} and multiply both sides by $$\Pi^N_a(\eta^N){(\eta^N(x))^\alpha}=\Pi^{N}_a(\eta^{N,x,y}){(\eta^{N,x,y}(y))^\alpha}.$$ By changing variables in the second term, involving $f_N(\eta^{N,x,y})$, and observing that $\theta(\eta^{N,x,y},\eta^N)=-\theta(\eta^N, \eta^{N,x,y})$, both terms on the left-hand side give the same contribution, so we may consider only the first term. Repeating the calculations in \cite[Lemma 6.1]{gess2023rescaled}, the contribution of the first term is bounded below by \begin{equation}
		\begin{split}
			& -\int_{X_N}f_N(\eta^N)(\eta^N(y)^{\alpha/2}-\eta^N(x)^{\alpha/2})(\eta^N(x))^{\alpha/2}\Pi^N_a(d\eta^N)  \\ & \hspace{2cm} -c\int_{X_N}f_N(\eta^{N})\chi_N^{\min(1,\alpha/2)}(1+(\eta^N(x))^\alpha+ (\eta^N(y))^{\alpha})\Pi^N_a(d\eta^N).
		\end{split}
	\end{equation}  Meanwhile, the summation over $\eta^N$ on the right-hand side is $\chi_N$ times the left-hand side of (\ref{eq: ruo}), which is therefore controlled by $\chi_Nk^2N^{-d}\d_N(f_{N})$, and the proof of (\ref{eq: reg by path conc}) is complete. \\ \\ The argument for (\ref{eq: reg by path conc'}) is similar. Let us first fix $x\in B, y\in B'$, and note that the lattice distance between them is at most $\widetilde{l}_N$ by hypothesis. Returning to (\ref{eq: basic inequality for dirichlet}), we now take \begin{equation}
		\theta=\frac\alpha2\left(\log(\Lambda_{B'}\eta^{N,x,y}) -\log (\Lambda_B\eta^N)\right) 
	\end{equation} so that, multiplying by $(\eta^N(x))^\alpha \Pi^N_a(d\eta^N)$ and integrating the left-hand side contributes \begin{equation}
		-2\int_{X_N}f_N(\eta^N)(\eta^N(x))^\alpha \left(\frac{(\Lambda_{B'}\eta^{N,x,y})^{\alpha/2}}{(\Lambda_B \eta^N)^{\alpha/2}}-1\right)\Pi^N_a(d\eta^N).
	\end{equation} In the numerator, we use that \begin{equation}
		(\Lambda_{B'}\eta^{N,x,y})^\alpha = (\Lambda_{B'}\eta^N)^\alpha + \frac{w^N(y)}{w^N(B')}\left((\eta^N(y)+\chi_N)^\alpha-(\eta^N(y))^\alpha\right)
	\end{equation} and the elementary inequality $\sqrt{a+b}\le \sqrt{a}+\sqrt{b}, a, b\ge 0$, to bound the whole contribution below by \begin{equation} \label{eq: box averaging}\hspace{-1.5cm}
		-2\int_{X_N}f_N(\eta^N)(\eta^N(x))^\alpha\left(\frac{(\Lambda_{B'}\eta^N)^{\alpha/2}+(w^N(y)/w^N(B'))^{1/2}\chi_N^\delta(1+(\eta^N(y))^{\alpha/2})}{(\Lambda_B \eta^N)^{\alpha/2}}-1\right)\Pi^N_a(d\eta^N).
	\end{equation} Meanwhile, the integral of the right-hand side is bounded by, using Lemma \ref{lemma: really useful observation}, $\chi_N\widetilde{l}_N^2N^{-d}\mathfrak{d}_N(f_N)$, which is therefore an upper bound for the expression in (\ref{eq: box averaging}). \\ \\ We now average this inequality over $x\in B, y\in B'$ with respect to the weighting $w^N$. The averaging in $x\in B$ replaces the occurrences of $(\eta^N(x))^\alpha, (\eta^N(y))^\alpha$ by $(\Lambda_B \eta^N)^\alpha, (\Lambda_{B'} \eta^N)^\alpha$ respectively, while the bound $\sup w^N\le A \inf w^N$ implies that $w^N(y)\le Al_N^{-d}w^N(B')$. All together, for some constant $C$,  \begin{equation}\begin{split}
		&-2\mathbb{E}_{f_N\Pi^N_a}\left[(\Lambda_B\eta^N)^{\alpha/2}\left((\Lambda_{B'}\eta^N)^{\alpha/2}-(\Lambda_B\eta^N)^{\alpha/2}\right)\right]\\& \hspace{1cm} -CA^{1/2}\chi_N^\delta l_N^{-d/2}\mathbb{E}_{f_N\Pi^N_a}\left[(1+(\Lambda_B\eta^N)^\alpha+(\Lambda_{B'}\eta^N)^\alpha\right] \le \widetilde{l}_N^2N^{-d}\mathfrak{d}_N(f_N)\end{split}
	\end{equation} which may be post-processed into the claimed inequality as in the previous case.   \end{proof}

\subsection{Integrability Estimate via Scale Decomposition}  \label{sec: int main} We now arrive at the proof of Proposition \ref{prop: int dirichlet form}, and hence Theorem \ref{thrm: main integrability estimate}, following the argument set out in Section \ref{sec: sketch}. For ease of readability, the proof is divided into the following Lemma, which gives one step of an iterative argument, and the subsequent assembly into an iterative scheme. \begin{lemma}[One Step of Multiscale Argument via Dirichlet form]\label{lemma: 1sr} Let $1\le l_N \le \widetilde{l}_N\le N$, and let $\cP_N\in \mathcal{Q}(N,l_N), \widetilde{\cP}_N\in \mathcal{Q}(N,\widetilde{l}_N)$ with $\cP_N\prec \widetilde{\cP}_N$, and let $w^N$ be a weight function satisfying the bound $\frac{\sup w^N}{\inf w^N}\le A$. Assume that, for all $B, B'\in \cP_N, w^N(B)=w^N(B')$, and similarly, for all $\widetilde{B}, \widetilde{B}'\in\widetilde{\cP_N}$, $w^N(\widetilde{B})=w^N(\widetilde{B}')$. For $p=p(d)>1$ given by Proposition \ref{prop: discrete Sobolev}, let $\Phi$ be a Young function such that \eqref{eq: Young-p consistency} holds for this $p$, and let $\Psi$ be the dual function. Then, for any an absolute constant $C=C(d,\Phi)$, $\delta=\delta(\alpha)>0$, for any $\lambda\in(0,1]$ and for any translationally invariant $f_N\in G_N(\Pi^N_a)$, it holds that \begin{equation}
		\label{eq: 1sr conclusion} \begin{split}\mathbb{E}_{f_N\Pi^N_a}\left[\|(\Lambda_{B}\eta^N)^\alpha\|_{\ell_\Phi(B\in \cP_N,w^N)} \right] & \le (1+\lambda)\mathbb{E}_{f_N\Pi^N_a}\left[\|(\Lambda_{\widetilde{B}_N}\eta^N)^\alpha\|_{\ell_\Phi(\widetilde{B}\in \widetilde{\cP}_N,w^N)} \right] \\[1ex] & \hspace{-2.5cm}+ C\lambda^{-1}A^{1/2}l_N^{-d/2}\left(\chi_N^{\delta/2} \Psi^{-1}(N^d)\right)(\chi_N^{\delta/2}\widetilde{l}_N^2 l_N^{-2}) \mathbb{E}_{f_N\Pi^N_a}\left[1+\|(\eta^N)^\alpha\|_{\ell^1(\TND)} \right] \\ &\hspace{-2.5cm}+C(\lambda^{-1}\widetilde{l}_NN^{-1})\left(\widetilde{l}_NN^{-1} \Psi^{-1}(N^d\widetilde{l}_N^{-d})\right)\chi_N N^{2-d}\mathfrak{d}_N(f_N).\end{split}
	\end{equation}
	%FINDME
		
	\end{lemma} \begin{remark} 
	%Let us motivate this estimate. By convexity, the averaging over smaller boxes $\cP_N\prec \widetilde{\cP}_N$ will (typically) behave disfavourably compared to the coarser averaging. This lemma allows us to pass from one scale $\widetilde{l}_N$ to another $l_N\ll \widetilde{l}_N$, adding an error term in the second line multiple of the Dirichlet form in the third line, and the key idea will be to apply this result on a sequence of scales $1\sim l^0_N\ll l^1_N\dots \ll l^k_N=N$. 
	This is the full version of the estimate \eqref{eq: one step toy}, now in sufficient generality to cover the argument without the simplifying hypotheses. As discussed in the introduction, an additional parameter $\lambda\in (0,1)$ has been introduced, in order to keep the prefactor of the iterative term bounded if the number of iterations is allowed to diverge. The various prefactors in the second and third lines are bracketed corresponding to the ways that the various compensations will be leveraged. \end{remark} 
	\begin{proof}
		We divide the proof into a probabilistic step, arguing from the Dirichlet form $\mathfrak{d}_N(f_N)$, and an analytic step. 
		\paragraph{\textbf{Step 1: One Block of $\widetilde{\cP}_N$ via Dirichlet Form}} Let us fix $\widetilde{B}\in \widetilde{\cP}_N$. In this step, we will show a \emph{local} version of (\ref{eq: 1sr conclusion}). Let us define \begin{equation}
			\label{eq: Delta Btilde}\Delta_{\widetilde{B}}(\eta^N):=\left[\left\|(\Lambda_{B}\eta^N)^\alpha \right\|_{\ell^p(B\in\cP_N|_{\widetilde{B}},w^N)}-(1+\lambda)(\Lambda_{\widetilde{B}}\eta^N)^\alpha\right]_+
		\end{equation} with $_+$ denoting the nonnegative part of a real number. We will show, in this step, that for any $f\in G_N(\Pi^N_a)$, \begin{equation}
			\label{eq: 1sr local} \begin{split}
		 \mathbb{E}_{f_N\Pi^N_a}[\Delta_{\widetilde{B}}(\eta^N)] \le &   C\lambda^{-1}\widetilde{l}_N^2N^{-d} \mathfrak{d}_N(f_N)  \\&  +CA^{1/2}\lambda^{-1}{\widetilde{l}_N^2}{l_N^{-2-d/2}}\chi_N^\delta \mathbb{E}_{f_N\Pi^N_a}\left[1+\|(\eta^N)^\alpha\|_{\ell^1(\TND)}\right]\end{split}
		\end{equation} where $\delta>0$ is the same exponent as in Lemma \ref{lemma: reg by paths}. Let us fix $f_N\in G_N(\Pi^N_a)$. For any neighbouring boxes $B, B'\subset \widetilde{B}$, $B, B'\in P_N$, we recall from Lemma \ref{lemma: reg by paths}, it holds that \begin{equation} \begin{split}
			\hspace{-0.5cm}\mathbb{E}_{f_N\Pi^N_a}\left[\left((\Lambda_B \eta^N)^{\alpha/2}-(\Lambda_{B'}\eta^N)^{\alpha/2}\right)^2\right] & \le cl_N^2N^{-d}\mathfrak{d}_N(f_N) \\ & +CA^{1/2}l_N^{-d/2}\chi_N^\delta\mathbb{E}_{f_N\Pi^N_a}(1+\|(\eta^N)^\alpha\|_{\ell^1(\TND)})		\end{split}
		\end{equation}  where we have used translation invariance of $f_N$ to simplify the error term. Averaging over pairs of adjacent $B, B' \in \cP_N, B, B'\subset \widetilde{B}$, we find \begin{equation}\begin{split} \label{eq: lsr local 1}
			& \mathbb{E}_{f_N\Pi^N_a}\left[\left(\frac{l_N}{\widetilde{l}_N}\right)^{d-2}\sum_{B\sim B'}\left((\Lambda_B \eta^N)^{\alpha/2}-(\Lambda_{B'}\eta^N)^{\alpha/2}\right)^2\right] \\ & \hspace{1cm}\le C\widetilde{l}_N^2 N^{-d} \mathfrak{d}_N(f_N)+CA^{1/2}l_N^{-d/2}\left(\frac{\widetilde{l}_N}{l_N}\right)^2\chi_N^\delta\mathbb{E}_{f_N\Pi^N_a}\left[1+\|(\eta^N)^\alpha\|_{\ell^1(\TND)}\right].
		\end{split}\end{equation} Since the graph formed by the adjacency of boxes $B\subset \widetilde{B}, B\in \cP_N$ is a lattice, with side lengths $\sim \widetilde{l}_N/l_N$, we now use the discrete Sobolev-Gagliardo-Nirenberg inequality in Proposition \ref{prop: discrete Sobolev}, for some $p>1$, $C=C(d,p)<\infty$, and all $\lambda\in (0,1)$, \begin{equation} \label{eq: discrete sobolev}
			\begin{split}
					\hspace{-1cm}\left\|(\Lambda_{B}\eta^N)^\alpha \right\|_{\ell^p(B\in  \cP_N|_{\widetilde{B}},w^N)} \le C\lambda^{-1}\left(\frac{l_N}{\widetilde{l}_N}\right)^{d-2}\sum_{B\sim B'}\left((\Lambda_B \eta^N)^{\alpha/2}-(\Lambda_{B'}\eta^N)^{\alpha/2}\right)^2 \\+(1+\lambda)\frac{1}{\#(\cP_N|_{\widetilde{B}})}\sum_{B\in \cP_N: B\subset \widetilde{B}} (\Lambda_B\eta^N)^{\alpha}.
			\end{split}
		\end{equation}Since every $B\in \cP_N$ receives equal weight, the final term is exactly $(\Lambda_{\widetilde{B}}\eta^N)^\alpha$. Taking the second term onto the left-hand side, we see that \begin{equation} \label{eq: discrete sobolev'}
			\begin{split}
					\hspace{-1cm}\Delta_{\widetilde{B}}(\eta^N) \le C\lambda^{-1}\left(\frac{l_N}{\widetilde{l}_N}\right)^{d-2}\sum_{B\sim B'}\left((\Lambda_B \eta^N)^{\alpha/2}-(\Lambda_{B'}\eta^N)^{\alpha/2}\right)^2.
			\end{split} %FINDME
		\end{equation} Taking expectations of (\ref{eq: discrete sobolev'}) and combining with (\ref{eq: lsr local 1}) produces (\ref{eq: 1sr local}). \\\paragraph{\textbf{Step 2. From One Block to Global Estimate.}} We now argue why (\ref{eq: 1sr local}) implies the conclusion (\ref{eq: 1sr conclusion}). First, thanks to Lemma \ref{lemma: orlicz consistency}, \begin{equation} \begin{split} \|(\Lambda_{B}\eta^N)^\alpha \|_{\ell_\Phi(B\in\cP_N,w^N)} & \le \left\|\|(\Lambda_{B}\eta^N)^\alpha\|_{\ell^p(B\in \cP_N|_{\widetilde{B}},w^N)}\right\|_{\ell_\Phi(\widetilde{B}\in \widetilde{\cP}_N,w^N)} \\ & \le \left\|\Delta_{\widetilde B}(\eta^N)+(1+\lambda)(\Lambda_{\widetilde{B}}\eta^N)^\alpha\right\|_{\ell_\Phi(\widetilde{B}\in \widetilde{\cP}_N,w^N)}.\end{split}\end{equation} Using the triangle inequality, the second term is one of the terms appearing in the right-hand side of (\ref{eq: 1sr conclusion}). For the term involving $\Delta_{\widetilde{B}}(\eta^N)$, we use the equivalence of the Orlicz and $\ell^1$ norms on finite sets to estimate $$\left\|\Delta_{\widetilde{B}}(\eta^N)\right\|_{\ell_\Phi(\widetilde{B}\in \widetilde{\cP}_N,w^N)}\le C\Psi^{-1}(N^d\widetilde{l}_N^{-d})\left\|\Delta_{\widetilde{B}_N}(\eta^N)\right\|_{\ell^1(\widetilde{B}\in\widetilde{\cP}_N,w^N)} $$ where, as in the statement of the theorem, $\Psi$ is the dual Young function to $\Phi$, and $C$ is an absolute constant. We therefore take expectations to find \begin{equation} \begin{split} 
			\mathbb{E}_{f_N\Pi^N_a}\left\|\Delta_{\widetilde{B}}(\eta^N)\right\|_{\ell_\Phi(\widetilde{B}\in\widetilde{\cP}_N,w^N)} \le C \Psi^{-1}(N^d\widetilde{l}_N^{-d})\sum_{\widetilde{B}\in \widetilde{\cP}_N} \frac{w^N(\widetilde{B})}{w^N(\TND)}\mathbb{E}_{f_N\Pi^N_a}[\Delta_{\widetilde{B}}(\eta^N)].   \end{split}
		\end{equation} Each term in the sum is bounded by Step 1, and the summation over $\widetilde{\cP}_N$  yields the claim (\ref{eq: 1sr conclusion}). \end{proof} In order to apply the previous result iteratively, we need to specify a suitable Young function $\Phi$, and a suitable sequence of partitions and weights. These are dealt with by the next two results, whose proofs we defer until Appendix \ref{subsec: orlicz constructions}. \begin{proposition}\label{prop: Young function}
			Fix any sequence $\chi_N\to 0$ and $\delta>0$. Then there exists a strict Young function $\Phi$ satisfying the hypotheses of Lemma \ref{lemma: 1sr}, and whose dual function $\Psi$ additionally satisfies \begin{equation} \label{eq: choice of Phi} \sup_{u\ge 1}u^{-1/d}\Psi^{-1}(u) <\infty; \qquad \sup_N \chi_N^{\delta/2} \Psi^{-1}(N^d)<\infty.\end{equation}		\end{proposition} 
			
			\begin{lemma}\label{lemma: p and w} Fix $\chi_N\to 0$ and $\delta>0$. Then, for every $N$, there exists a sequence of length scales $l^k_N$, partitions $\cP^1_N\prec \cP^2_N\prec \dots \prec \cP^{K_N}_N$, with $\cP^k_N\in \mathcal{Q}(N, l^k_N)$, and a weighting $w^N$ such that: \begin{enumerate}
				\item For every $1\le k\le K_N$, for every $B_k\in \cP^k_N$, $w^N(B_k)=|\cP^k_N|^{-1}w^N(\TND)$, so each block of $\cP^k_N$ receives the same weight; \item  	The initial scale $l^1_N=1, l^{K_N}_N\ge \frac{N}{2}$, and $\cP^{K_N}_N$ is the trivial partition $\cP^{K_N}_N=\{\TND\}$, and\begin{equation}
					\label{eq: consistency of scales} \sup_N \max_{k<K_N} \chi_N^{\delta/2} \left(\frac{l^{k+1}_N}{l^k_N}\right)^2 < \infty; \qquad \min_{k<K_N} \min\left(\chi_N, \frac12\right)^{\delta/4} \left(\frac{l^{k+1}_N}{l^k_N}\right)^2 \ge \frac14; 
				\end{equation}\item For some absolute constant $A$, for all $N$, $\frac{\sup w^N}{\inf w^N} \le A$.		\end{enumerate}   \end{lemma} 
				
				\begin{proof}[Proof of Proposition \ref{prop: int dirichlet form}] Fix $\chi_N\to 0$ as in the statement, let $p>1$ be given by Proposition \ref{prop: discrete Sobolev}, and $\delta=\delta(\alpha)=\min(1,\frac\alpha2)$. With these choices, let us take $\Phi$ to be the Young function produced by Proposition \ref{prop: Young function}. We will now prove that, for this choice of $\Phi$, for any $a, b>0$, and for some sufficiently large $\vartheta\in (0,\infty)$, the estimate \eqref{eq: main integrability conclusion'}, which asserts that \begin{equation}
					\label{eq: main proof conclusion'} \limsup_N\sup_{f_N}\left\{\mathbb{E}_{f_N\Pi^N_a}\left[\|(\eta^N)^\alpha\|_{\ell_\Phi}1(\eta^N\in X_{N,b})\right]-\vartheta \chi_N N^{2-d}\mathfrak{d}_N(f_N)\right\}<\infty.
				\end{equation}  Thanks to Lemma \ref{lemma: gallilean pdfs}, we can restrict to $f_N\in G_N(\Pi^N_a,b)$. We proceed by using Lemma \ref{lemma: 1sr} repeatedly on the partitions $\cP^1_N\prec \cP^2_N\prec\dots \prec \cP^{K_N}_N$ and weight $w^N$ constructed by Lemma \ref{lemma: p and w}, and set $\cP^0_N:=\TND$.  \\ \paragraph{\emph{Step 1. From $\cP^0_N$ to $\cP^1_N$.}} In order to start the iteration, we must first convert the `standard' partition $\TND$ and weight $w=1$, for which the Orlicz norm in the Theorem is defined, into the partition $\cP^1_N$ by microscopic boxes and weight $w^N$. We will show that for some absolute constant $c$ and any $\eta^N\in X^N$, \begin{equation} \label{eq: compare Orlicz small scales}
					\|(\eta^N)^\alpha\|_{\ell_\Phi} \le c\|(\Lambda_{B}\eta^N)^\alpha\|_{\ell_\Phi(B\in\cP^1_N,w^N)}
				\end{equation} where we recall that the left-hand side is defined with respect to the uniform weight $w=1$, while the right-hand side is defined with the weight $w^N$ from Lemma \ref{lemma: p and w}. For any $x\in \TND$, let $B\in \cP^1_N$ be the element containing $x$; since $l^1_N$ and $\frac{\sup w^N}{\inf w^N}$ are bounded uniformly in $N$, $$ (\eta^N(x))^\alpha \le c(\Lambda_{B}\eta^N)^\alpha  $$ for some $c$ independent of $N$. Hence, for any $t\ge 0$, \begin{equation}
					\frac{1}{N^d}\sum_{x\in \TND} \Phi\left(\frac{(\eta^N(x))^\alpha}{t}\right) \le \frac{w^N(\TND)}{N^d}\sum_{B\in \cP^1_N} \frac{|B|}{w^N(B)} \frac{w^N(B)}{w^N(\TND)} \Phi\left(\frac{c(\Lambda_B\eta^N)^\alpha}{t}\right).
				\end{equation} The two factors $\frac{w^N(\TND)}{N^d}\frac{|B|}{w^N(B)}$ are bounded above by $\frac{\sup w^N}{\inf w^N}\le A$, and by convexity the right-hand side is now at most $$ \frac{1}{N^d}\sum_{x\in \TND}\Phi\left(\frac{(\eta^N(x))^\alpha}{t}\right) \le \sum_{B\in \cP^1_N} \frac{w^N(B)}{w^N(\TND)} \Phi\left(\frac{cA(\Lambda_B\eta^N)^\alpha}{t}\right).$$ Setting $t:=cA\|(\Lambda_{B}\eta^N)^\alpha\|_{\ell_\Phi(B\in\cP^1_N,w^N)}$ makes the right-hand side at most 1, and \eqref{eq: compare Orlicz small scales} is proven for a new choice of $c$. \\ \paragraph{\emph{Step 2. Iterative Scheme}} We now apply the iterative scheme for which Lemma \ref{lemma: 1sr} gives us the basic step. Let us fix $f_N\in G_N(\Pi^N_a)$, and to shorten notation let us write $$ Z_k(f_N):=\mathbb{E}_{f_N\Pi^N_a}\left\|(\Lambda_{B}\eta^N)^\alpha\right\|_{\ell_\Phi(B\in\cP^k_N,w^N)}, \qquad 1\le k\le K_N.$$ We now write the telescoping sum, for a sequence of $\lambda_k$ to be chosen later,  \begin{equation} \begin{split}
					\label{eq: telescoping sum} Z_1(f_N) & = \sum_{k=1}^{K_N-1} \left(\prod_{1\le m<k}(1+\lambda_m)\right)\left(Z_k(f_N)-(1+\lambda_k)Z_{k+1}(f_N)\right) \\ & \qquad \qquad +\left(\prod_{1\le m\le K_N} (1+\lambda_m) \right)Z_{K_N}(f_N).\end{split}
				\end{equation}  Since each $\cP^k_N \prec \cP^{k+1}_N$ satisfies the hypotheses of Lemma \ref{lemma: 1sr}, we may apply the result to each term of the summation to find, for an $N$-independent constant $C$, \begin{equation}\label{eq: form of bound}
					\begin{split} Z_1(f_N) & \le C\mathcal{T}^1_N(1+\EE_{f_N\Pi^N_a}\|(\eta^N)^\alpha\|_{\ell^1}) + C\mathcal{T}^2_N (\chi_N N^{2-d} \mathfrak{d}_N(f_N)),\end{split} 
				\end{equation} where $\mathcal{T}^1_N$ collects the different error terms and the final term of the iteration \begin{equation}\begin{split}  \label{eq: T1} \mathcal{T}^1_N:&= \sum_{k=1}^{K_N-1} \left(\prod_{1\le m<k}(1+\lambda_m)\right) \lambda_k^{-1}(l^k_N)^{-d/2}\{\chi_N^{\delta/2}\Psi^{-1}(N^d)\}\left[\chi_N^{\delta/2}\left(\frac{l^{k+1}_N}{l^k_N}\right)^2\right] \\ & \qquad \qquad + \left(\prod_{1\le m\le K_N} (1+\lambda_m)\right), \end{split} \end{equation} while $\mathcal{T}^2_N$ collects all the terms appearing in front of the Dirichlet form: \begin{equation}
					\label{eq: T2} \mathcal{T}^2_N:= \sum_{k=1}^{K_N-1}\left(\prod_{1\le m<k}(1+\lambda_m)\right)(\lambda_k^{-1}l^{k+1}_NN^{-1})\{l^{k+1}_NN^{-1} \Psi^{-1}(N^d(l^{k+1}_N)^{-d})\}.
				\end{equation} Here, we have made two simplifications. Firstly, since factors relating to $A$ are bounded uniformly in $N$, they are absorbed into the constant $C$. Secondly, by definition of the $\alpha$-averaging, the final term of the iteration $Z_{K_N}(f)$ is comparable to $\EE_{f_N\Pi^N_a}\|(\eta^N)^\alpha\|_{l^1}$. \\ \\ We now simplify the previous two expressions \eqref{eq: T1} - \eqref{eq: T2} using the construction of $\Phi$ and $(l^k_N)_{1\le k< K_N}$. In both expressions, the factor in $\{\cdot\}$ involving $\Psi^{-1}$ is bounded, uniformly in $N$, by \eqref{eq: choice of Phi}, since $\frac{N}{l^{k+1}_N}\ge 1$. Meanwhile, the construction \eqref{eq: consistency of scales} ensures that the factor $[\cdot ]$ appearing in $\mathcal{T}^1_N$ is bounded, uniformly in $k, N$. Absorbing these $N$-independent factors into the constant $C$, we have a bound of the same form \eqref{eq: form of bound} with the new terms \begin{equation}
					\label{eq: T1'}  \mathcal{T}^1_N:= \sum_{k=1}^{K_N-1} \left(\prod_{1\le m<k}(1+\lambda_m)\right) \lambda_k^{-1}(l^k_N)^{-d/2}+ \left(\prod_{1\le m\le K_N} (1+\lambda_m)\right)
				\end{equation} and \begin{equation}
					\label{eq: T2'} \mathcal{T}^2_N:= \sum_{k=1}^{K_N-1}\left(\prod_{1\le m<k}(1+\lambda_m)\right)(\lambda_k^{-1}l^{k+1}_NN^{-1}).
				\end{equation} \paragraph{\emph{Step 3. Choice of $\lambda_k$}} We must now specify the sequence $\lambda_k\in (0,1], 1\le k<K_N$ in the previous step. We will now show that choosing \begin{equation} \label{eq: magic choice of lambda} \lambda_k:=\max\left((l_N^k)^{-d/4}, \left(\frac{l^{k+1}_N}{N}\right)^{1/2}\right) \end{equation} produces upper bounds on both expressions \eqref{eq: T1'}-\eqref{eq: T2'} which are uniform in $N$. It is immediate that $0<\lambda_k\le 1$, and since $l^{k+1}_N\ge \frac12 \min(\frac12,\chi_N)^{-\delta/8} l^k_N$, we get the two bounds $$(l^k_N)^{-d/4} \le \min\left(\frac14,\frac{\chi_N}{2}\right)^{\delta d k/32},\qquad  \left(\frac{l^{k}_N}{N}\right)^{1/2} \le \min\left(\frac12, \chi_N\right)^{\delta (K_N-k)/16}$$ which is summable in $k$, uniformly in $N$. As a result, the sums $\sum_{k=1}^{K_N-1} (l^k_N)^{-d/4}$,  $\sum_{k=1}^{K_N-1}(l^{k+1}_N/N)^{1/2}$ are bounded, uniformly in $N$. From the definition of $\lambda_k$, there exists a $N$-independent $C$ such that $\sum_{k=1}^{K_N-1} \lambda_k \le \log C$, and for the products $$ \max_{1\le k<K_N}\prod_{1\le m<k} (1+\lambda_m)\le C $$ is similarly bounded. These bounds also imply $N$-uniform bounds on $$ \sum_{k=1}^{K_N-1} \lambda_k^{-1}(l^k_N)^{-d/2}; \qquad \sum_{k=1}^{K_N-1} \lambda_k^{-1}\left(\frac{l^{k+1}_N}{N}\right)$$ and the step is complete. \\ \paragraph{\emph{Step 4. Assembly \& Interpolation}} Assembling the previous steps, for some $N$-independent $C_1, C_2, C_3$, \begin{equation}
					\begin{split} \mathbb{E}_{f_N\Pi^N_a}\left[\|(\eta^N)^\alpha\|_{\ell_\Phi}\right] & \le C_1 \mathbb{E}_{f_N\Pi^N_a}\left[\|(\Lambda_{\cP^1_N}\eta^N)^\alpha\|_{\ell\Phi(B\in \cP^1_N,w^N)}\right] \\ & \le C_2\mathbb{E}_{f_N \Pi^N_a}\left[1+\|(\eta^N)^\alpha\|_{\ell^1}\right] +C_3\chi_N N^{2-d}\mathfrak{d}(f_N)  \end{split}
				\end{equation} where the first line follows from Step 1, and the second by applying Steps 2-3 to \eqref{eq: form of bound}. In order to interpolate to the desired bound \eqref{eq: main integrability conclusion}, we must make use of the support condition $f_N\Pi^N_a(X_{N,b})=1$ which we assumed at the start of the proof. Since $\Phi$ is strict, we may use Lemma \ref{lemma: orlicz interpolation} to find $z$, depending only $\Phi, b$ and the absolute constant $C_2$ in the last expression, such that for $f_N\Pi^N_a$-almost all $\eta^N$, \begin{equation}
				\|(\eta^N)^\alpha\|_{\ell^1}\le 	\frac{1}{2C_2}\|(\eta^N)^\alpha\|_{\ell_\Phi}+\frac{z}{C_2}
				\end{equation} which allows us to eliminate the $\ell^1$-term from the right hand side to find \begin{equation}
					\begin{split} \mathbb{E}_{f_N\Pi^N_a}\left[\|(\eta^N)^\alpha\|_{\ell_\Phi}\right]  & \le 2(C_2+z +C_3\chi_N N^{2-d}\mathfrak{d}(f_N)).  \end{split}
				\end{equation} If we now take $\vartheta:=2C_3$, we have proven the claim \eqref{eq: main proof conclusion'} when the supremum is restricted to $f\in G_N(\Pi^N_a,b)$. As argued above, the supremum is the same with and without this restriction, and the theorem is complete.  \end{proof}

\section{Exponential Tightness} \label{sec: ET}
We now prove Proposition \ref{prop: tightness}. \begin{proof}[Proof of Proposition \ref{prop: tightness}] We will prove that \begin{equation}\label{eq: ET 1} \limsup_{M\to \infty}\limsup_{N\to \infty} \frac{\chi_N}{N^d}\log \mathbb{P}_{\Pi^N_a}\left(\exists t\le T: \cH(\eta^N_t)>M\right)=-\infty\end{equation} and, for all $\epsilon>0$, there exists $\delta>0$ such that \begin{equation} \label{eq: ET 2}\limsup_{\delta\to 0} \limsup_{N\to \infty}\frac{\chi_N}{N^d}\log \mathbb{P}_{\Pi^N_a}\left(\exists s, t\in [0,T]: 0< s-t<\delta, \int_s^t \|(\eta^N_u)^\alpha\|_{\ell^1} du>\epsilon\right)=-\infty. \end{equation} Once these are in hand, the same argument as in \cite[Lemma 6.1]{gess2023rescaled} shows that $\eta^N_\bullet$ fulfil conditions of the general abstract result \cite[Theorems 4.1, 4.4]{feng2006large} and the result is complete.
	\paragraph{\emph{Step 1. Entropy Bound}} We first prove \eqref{eq: ET 1}. From \cite[Lemma 4.4]{gess2023rescaled} it holds, for all $0\le \beta<\alpha$, \begin{equation}
		\label{eq: entropy static} \limsup_{N\to \infty} \frac{\chi_N}{N^d}\log \mathbb{E}_{\Pi^N_a}\left[\exp\left(\frac{N^d}{\chi_N}\beta\cH(\eta^N_0)\right)\right]<\infty.
	\end{equation} Fix $0\le z<\infty$. Thanks to \eqref{eq: entropy static}, we may choose $b<\infty$ such that \begin{equation} \begin{split}
		\label{eq: choose b} \limsup_N \frac{\chi_N}{N^d}\log \mathbb{P}_{\Pi^N_a}\left(\exists t\le T: \eta^N_t\not \in X_{N,b}\right)&=\limsup_N \frac{\chi_N}{N^d}\log \mathbb{P}_{\Pi^N_a}\left(\eta^N_0\not \in X_{N,b}\right) \\ & <-z. \end{split}
	\end{equation} On the event $\eta^N_0\in X_{N,b}$, it follows that, almost surely, $\eta^N_t(x)\le N^d b$ for all $x\in \TND$ and all $t\ge 0$ thanks to the conservation of mass, whence the total jump rate is at most \begin{equation} \label{eq: crude bound on rate} \frac{N^{2-d}}{\chi_N}\sum_{x\in \TND} (\eta^N(x))^\alpha \le \frac{N^{2+(\alpha-1)d}b^\alpha}{\chi_N}.\end{equation} Let now $\tau^N_k, k\ge 1$ be the jump times of $\eta^N_t$ and let $Z^N(t)$ be the process counting the jumps. Thanks to the finiteness of the maximum rate, $Z^N(t)$ is dominated by a Poisson process of intensity given by \eqref{eq: crude bound on rate}, and in particular there exists a $c=c(z)$,  and $\gamma>0$ such that \begin{equation}\label{eq: total no of jumps}
		\limsup_{N\to \infty} \frac{\chi_N}{N^d} \log \PP_{\Pi^N_a}\left(Z^N(T)> c\left(\frac{N^d}{\chi_N}\right)^\gamma, \eta^N_0\in X_{N,b}\right) \le -z. 
	\end{equation} Finally, on the event $Z^N(T)\le c(N^d/\chi_N)^\gamma$, the values taken by $\eta^N_t$ on $t\in [0,T]$ are a subset of $\{\eta^N_{\tau^N_k}: k\le  c(N^d/\chi_N)^\gamma\}$, and for any $0<\beta<\alpha$, \begin{equation}\begin{split} 
		&\mathbb{E}_{\Pi^N_a}\left[1(Z^N(t)\le c(N^d/\chi_N)^\gamma) \exp\left(\frac{N^d \beta \sup_{t\le T}\cH(\eta^N_t)}{\chi_N}\right)\right] \\ & \le \EE_{\Pi^N_a}\left[\sum_{k=1}^{c(N^d/\chi_N)^\gamma}\exp\left(\frac{N^d \beta \cH(\eta^N_{\tau^N_k})}{\chi_N}\right)\right] = c\left(\frac{N^d}{\chi_N}\right)^\gamma \EE_{\Pi^N_a}\left[\exp\left(\frac{N^d \beta \cH(\eta^N_{0})}{\chi_N}\right)\right]
	\end{split} \end{equation} where, in the last line, we use that each $\eta^N_{\tau^N_k}\sim \Pi^N_a$ has the equilibrium distribution. Thanks to \eqref{eq: entropy static}, the final factor is bounded above by $e^{\theta N^d/\chi_N}$, for some $\theta=\theta(a, \beta)$, and taking the logarithm and passing to the limit, the prefactor $(N^d/\chi_N)^\gamma$ does not contribute. By a Chebychev estimate, setting $M=\theta+z$, \begin{equation} \label{eq: third event ET}
		\limsup_N \frac{\chi_N}{N^d}\log \PP_{\Pi^N_a}\left(\sup_{t\le T} \cH(\eta^N_t)>M, Z^N(T)\le c(N^d/\chi_N)^\gamma \right) \le -z.
	\end{equation} The event $\{\sup_{t\le T} \cH(\eta^N_t)>M\}$ is now the union of the events \eqref{eq: choose b}, \eqref{eq: total no of jumps}, \eqref{eq: third event ET}, and the conclusion follows. \paragraph{\emph{Step 2. Time Integrability}} We now prove \eqref{eq: ET 2} with the help of Theorem \ref{thrm: main integrability estimate}. We again fix $z<\infty$, and choose $b<\infty$ as in \eqref{eq: choose b}. For this value of $b$, we choose $\beta>0$ and a strict Young function $\Phi$ according to Theorem \ref{thrm: main integrability estimate}, so that \begin{equation}
		\theta:=\sup_N \frac{\chi_N}{N^d}\log \mathbb{E}_{\Pi^N_a}\left[\exp\left(\frac{\beta N^d}{\chi_N}\int_0^T \|(\eta^N_s)^\alpha\|_{\ell_\Phi} 1(\eta^N_s\in X_{N,b}) ds\right)\right]<\infty
	\end{equation} is finite. Using Lemma \ref{lemma: orlicz interpolation}, we can choose $\lambda<\infty$ such that, for all $t\ge 0$, \begin{equation}
		\|(\eta^N_s)^\alpha\|_{\ell^1}1(\eta^N_s\in X_{N,b}) \le \lambda+ \frac{\beta \epsilon}{2(\theta +z)}\|(\eta^N_s)^\alpha\|_{\ell_\Phi}1(\eta^N_s\in X_{N,b})
	\end{equation} and set $\delta:=\frac{\epsilon}{2\lambda}$. It then follows that, for any time interval $[s,t]\subset [0,T]$ of length at most $\delta$, \begin{equation}
		\int_s^t \|(\eta^N_u)^\alpha\|_{\ell^1}1(\eta^N_u\in X_{N,b})du \le \frac{\epsilon}{2}+ \frac{\beta \epsilon }{2(\theta+z)} \int_0^T\|(\eta^N_u)^\alpha\|_{\ell_\Phi}1(\eta^N_u\in X_{N,b})du. 
	\end{equation} Hence, for any such interval to exist for which $\int_s^t \|(\eta^N_u)^\alpha\|_{\ell^1} du>\epsilon$, either the event in \eqref{eq: choose b} occurs, or $$ \int_0^T \|(\eta^N_u)^\alpha\|_{\ell_\Phi}1(\eta^N_u\in X_{N,b})du > \theta+z. $$ The final event has probability at most $e^{-N^d z/\chi_N}$ by a Chebychev estimate, and the proof is complete. 
\end{proof}

 \section{Superexponential Estimate} \label{sec: supex} We conclude with the proof of Theorem \ref{thrm: supex}. As promised in the introduction, the following proof avoids any problems due to the degeneration of the smallest jump rate $\chi_N^{\alpha-1}\to 0$ by using the pathwise regularity (in an averaged sense) introduced by Lemma \ref{lemma: reg by paths}. The rescaling of particle size by $\chi_N\to 0$ also provides a slight simplification, in that we go directly from a single point to a small macroscopic box, without the usual step of a large microscopic box. This is, however, not a crucial feature, and the following method would also apply in other cases with degenerate jump rates where the intermediate scale must be kept. \begin{proof}[Proof of Theorem \ref{thrm: supex}] Let us fix $z<\infty$, and choose $b<\infty$ such that \begin{equation} \label{eq: choose b 2}
	\limsup_{N\to \infty} \frac{\chi_N}{N^d}\log \PP_{\Pi^N_a}\left(\eta^N_0\not \in X_{N,b}\right)\le -z. 
\end{equation} Thanks to \eqref{eq: ET 1}, we may choose $\gamma>0$ such that \begin{equation} \limsup_{N\to \infty} \frac{\chi_N}{N^d} \log \PP_{\Pi^N_a}\left(\gamma T \sup_{t\le T} \cH(\eta^N_t)> \frac{\delta}{2} \right) \le -z. \end{equation} 
%It is now sufficient to prove \eqref{eq: supex conclusion} when the random variable $N^{-d}\sum_{x\in \TND} V^N_\alpha(\eta^N_t,\epsilon,x)$ is replaced by $$ \frac{1}{N^d}\sum_{x\in \TND} V^N_\alpha(\eta^N_t,\epsilon,x)1(\langle \eta^N_t, 1\rangle \le b)%-\lambda \cH(\eta^N_t).
%$$
Let us denote the average of \eqref{eq: VNA} over $x\in \TND$ by $V^N_{\alpha}(\eta^N,\epsilon)$. Using the Feynman-Kac formula \eqref{eq: FK formula}, it is sufficient to prove that, for all $\vartheta>0$, \begin{equation}
	\label{eq: supex after FK} \hspace{-0.5cm} \limsup_{\epsilon \to 0} \limsup_{N\to \infty} \sup_{f_N}\left\{\EE_{f_N\Pi^N_a}\left[V^N_\alpha(\eta^N, \epsilon)1(\eta^N_t\in X_{N,b}) - \gamma \cH(\eta^N)\right] - \vartheta \chi_N N^{2-d} \mathfrak{d}_N(f_N)\right\} \le 0.
\end{equation} Thanks to Lemma \ref{lemma: gallilean pdfs}, we may harmlessly restrict the supremum to $f_N\in G_N(\Pi^N_a,b)$. We first truncate, using Proposition \ref{prop: int dirichlet form}, and then use a single approximation corresponding to both one- and two-block estimates \cite{kipnis1998scaling}, based on Lemma \ref{lemma: reg by paths}. \\ \paragraph{\emph{Step 1. Truncation using Integrability Estimate}} In this step we will prove that, for any $\vartheta>0$ and $\lambda>0$, there exists $M<\infty$ such that, uniformly in $\epsilon>0$, \begin{equation}\label{eq: truncation}
	\limsup_{N\to \infty} \sup_{f_N \in G_N(\Pi^N_a, b)} \left\{\EE_{f_N\Pi^N_a}[|V^N_\alpha(\eta^N, \epsilon)-V^{N}_{\alpha,M}(\eta^N,\epsilon)|]-\frac{\vartheta}{2} \chi_N N^{2-d} \mathfrak{d}_N(f_N)\right\} \le \lambda
\end{equation} where $V^N_{\alpha,M}$ is defined by replacing every instance of the nonlinearity $\varphi(u)= u^\alpha$ in \eqref{eq: VNA} by the globally Lipschitz nonlinearity \begin{equation}\label{eq: varphi M} \varphi_M(u):=\begin{cases} u^\alpha, & u\le M; \\ M^{\alpha}+\alpha M^{\alpha-1}(u-M), &  u>M. \end{cases}\end{equation} We start by observing that, for any $\eta^N \in X_{N,b}$, $\epsilon$, \begin{equation} \label{eq: truncation start}
	|V^N_\alpha(\eta^N,\epsilon)-V^N_{\alpha,M}(\eta^N,\epsilon)| \le \|(\eta^N)^\alpha 1(\eta^N>M)\|_{\ell^1} +  \|(\overline{\eta}^{N,N\epsilon})^\alpha 1(\overline{\eta}^{N,N\epsilon}>M)\|_{\ell^1}.
\end{equation} We give an argument for the first term, which also treats the second term with only cosmetic modification. For $\Phi$ the Young function of Theorem \ref{thrm: main integrability estimate} and $\Psi$ the dual, we estimate, for $\eta^N\in X_{N,b}$, \begin{equation}\begin{split} 
	\label{eq: cutoff 1} \|(\eta^N)^\alpha1(\eta^N>M)\|_{\ell^1}&\le 2\|(\eta^N)^\alpha\|_{\ell_\Phi} \|1(\eta^N>M)\|_{\ell_\Psi}  \\& \le  \frac{2}{\Psi^{-1}(M/b)}\|(\eta^N)^\alpha\|_{\ell_\Phi} \end{split}
\end{equation} where we use the H\"older inequality for Orlicz norms \cite[Proposition 1.14]{leonard2007orlicz} in the first line, and the second line follows using the definition of $\|\cdot\|_{\ell_\Psi}$ and a Chebychev inequality. Thanks to Proposition \ref{prop: int dirichlet form} there exists $\vartheta_1, \lambda_1>0$, depending only on $a,b$, such that, for any $f_N\in G_N(\Pi^N_a, b)$, \begin{equation}
	\mathbb{E}_{f_N \Pi^N_a}\left[\|(\eta^N)^\alpha\|_{\ell_\Phi}\right] \le \vartheta_1 \chi_N N^{2-d} \mathfrak{d}_N(f_N) + \lambda_1.
\end{equation} We now choose $M$ large enough, depending on $a, b, \vartheta, \lambda$, but not on $N$, \begin{equation}
	\label{eq: choose M} \frac{2\vartheta_1}{\Psi^{-1}(M/b)} \le \frac{\vartheta}{4}; \qquad \frac{2\lambda_1}{\Psi^{-1}(M/b)} \le \frac{\lambda}{2}
\end{equation} so that, for any $f_N\in G_N(\Pi^N_a,b)$, \begin{equation}
	\EE_{f_N\Pi^N_a}\left[\|(\eta^N)^\alpha1(\eta^N>M)\|_{\ell^1}\right]-\frac{\vartheta}{2}\chi_N N^{2-d}\mathfrak{d}_N(f_N) \le \frac\lambda2.
\end{equation} The same argument applies to $\overline{\eta}^{N,\epsilon}$ in place of $\eta^N$, and returning to \eqref{eq: truncation start}, we have proven \eqref{eq: truncation}. \\ \paragraph{\emph{Step 2. One- and Two-Block Estimates}} We now simultaneously prove one- and two-block estimates, to show that $\eta^N(0)$ is close to $\bar{\eta}^{N, N\epsilon}(0)$. In this step, we prove that, for any $\vartheta>0$ and $\gamma>0$\begin{equation}\hspace{-1cm}
	\label{eq: one block conclusion} \limsup_{\epsilon\to 0, N\to \infty} \sup_{f_N\in G_N(\Pi^N_a,b)}\left\{\mathbb{E}_{f_N\Pi^N_a}[|\eta^N(0)-\overline{\eta}^{N,N\epsilon}(0)|-\gamma\cH(\eta^N)]-\frac{\vartheta}{2}\chi_N N^{2-d}\mathfrak{d}_N(f_N)\right\} \le 0.
\end{equation}  We start by bounding the supremum above by using translational invariance to write writing, for any $f_N$ and any $\lambda, \gamma>0$ \begin{equation} \begin{split} \label{eq: one block 0}
	&\mathbb{E}_{f_N\Pi^N_a}\left[|\eta^N(0)-\overline{\eta}^{N,N\epsilon}(0)| - \gamma \cH(\eta^N)\right] \\ & \le \mathbb{E}_{f_N \Pi^N_a}\left[\frac{1}{(2\lfloor \epsilon N\rfloor + 1)^d}\sum_{y\in B_{\lfloor \epsilon N \rfloor}(0)}|\eta^N(y)-\eta^N(0)|-\frac{\gamma}{2}\left\{w(\eta^N(0))+w(\eta^N(y))\right\}\right]
	\end{split} \end{equation}
	where we recall that $w(u)=u\log u$ is the entropy nonlinearity. Next, we use the elementary inequality that, for all $\alpha\ge 1, \gamma>0, \lambda>0$, there exists $C_{\alpha, \gamma,\lambda}<\infty$ such that, for all $x, y\ge 0$, \begin{equation}\label{eq: elementary inequality}
		|x-y|-\frac{\gamma}{2}(w(x)+w(y))\le \lambda + C_{\alpha, \gamma,\lambda}(x^{\alpha/2}-y^{\alpha/2})^2.
	\end{equation} We can thus bound \begin{equation} \begin{split}
	&\mathbb{E}_{f_N\Pi^N_a}\left[|\eta^N(0)-\overline{\eta}^{N,N\epsilon}(0)| - \gamma \cH(\eta^N)\right] \\ & \le \mathbb{E}_{f_N \Pi^N_a}\left[\lambda + \frac{C}{(2\lfloor \epsilon N\rfloor + 1)^d}\sum_{y\in B_{\lfloor \epsilon N \rfloor}(0)}|(\eta^N(y))^{\alpha/2}-(\eta^N(0))^{\alpha/2}|^2\right]
	\end{split} \end{equation}We now compute the expectation by the first part \eqref{eq: reg by path conc} of Lemma \ref{lemma: reg by paths} to find, up to a new value of $C$ with the same dependency, and $\delta=\delta(\alpha)>0$, \begin{equation}\label{eq: one block 1} \begin{split}\mathbb{E}_{f_N\Pi^N_a}\left[|\eta^N(0)-\overline{\eta}^{N,N\epsilon}(0)|-\gamma\cH(\eta^N)\right] &\le \lambda + C \chi_N \epsilon^2 N^{2-d} \mathfrak{d}_N(f_N) \\ & +C\chi_N^\delta \mathbb{E}_{f_N \Pi^N_a}[1+\|(\eta^N)^\alpha\|_{\ell^1}].\end{split} \end{equation} The expectation in the final line is bounded by $C(1+\chi_NN^{2-d}\mathfrak{d}_N(f_N))$ thanks to Proposition \ref{prop: int dirichlet form} and Lemma \ref{lemma: orlicz interpolation}, and all together, for all $f_N\in G_N(\Pi^N_a,b)$ and a new value of $C=C(\alpha, \gamma,\lambda)$, \begin{equation}\label{eq: one block 2} \begin{split}& \mathbb{E}_{f_N\Pi^N_a}\left[|\eta^N(0)-\overline{\eta}^{N,N\epsilon}(0)|-\gamma\cH(\eta^N)\right] - \frac{\vartheta}{2}\chi_N N^{2-d}\mathfrak{d}_N(f_N) \\  &\hspace{2cm} \le (\lambda+ C\chi_N^\delta) + \chi_N\left(C\chi_N^\delta+C\epsilon^2-\frac{\vartheta}{2}\right)N^{-d} \mathfrak{d}_N(f_N).\end{split} \end{equation} %As soon as $\epsilon<\sqrt{\vartheta/C}$, the term in parentheses is negative for all sufficiently large $N$. We may therefore optimise over $f_N$ and take $N\to \infty$ to find that \begin{equation}\begin{split} \label{eq: one block 3} \limsup_{N\to \infty} \sup_{f_N\in G_N(\Pi^N_a,b)} &\bigg\{\mathbb{E}_{f_N\Pi^N_a}\big[|\eta^N(0)-\overline{\eta}^{N,N\epsilon}(0)| -\gamma\cH(\eta^N)\big] \\ & \hspace{3cm}-\frac{\vartheta}{2}\chi_N N^{2-d}\mathfrak{d}_N(f_N)\bigg\} \le \lambda\end{split}\end{equation} 
	This bound is uniform in $f\in G_N(\Pi^N_a, b)$, and is immediate to see that the limit superior of the right-hand side in the double-limit $N\to \infty, \epsilon\to 0$ is at most $\lambda$. Since $\lambda>0$ was arbitrary, have proven \eqref{eq: one block conclusion}.\\ \paragraph{\emph{Step 3. Reassembly}} We finally reassemble the previous two steps to prove \eqref{eq: supex after FK}. Fix $\vartheta>0$ and $\lambda>0$. Thanks to the conclusion \eqref{eq: truncation} of Step 1, we can choose $M<\infty$ such that, uniformly in $\epsilon>0, N, f_N$, \begin{equation} \label{eq: reassembly 0}
		\EE_{f_N\Pi^N_a}[|V^N_\alpha(\eta^N, \epsilon)-V^{N}_{\alpha,M}(\eta^N,\epsilon)|]-\frac{\vartheta}{2} \chi_N N^{2-d} \mathfrak{d}_N(f_N) \le \frac{\lambda}{2}.
	\end{equation} For $M$ thus chosen, let $K_M$ be the Lipschitz constant of $\varphi_M$. Thanks to \eqref{eq: one block conclusion}, we may choose $\epsilon>0$ such that, for all sufficiently large $N$ and all $f_N\in G(\Pi^N_a, b)$, \begin{equation}
		\label{eq: reassembly} \mathbb{E}_{f_N\Pi^N_a}[|\eta^N(0)-\overline{\eta}^{N,N\epsilon}(0)|-\frac{\gamma}{K_M}\cH(\eta^N)]-\frac{\vartheta}{2K_M}\chi_N N^{2-d}\mathfrak{d}_N(f_N) \le \frac{\lambda}{2K_M}.
	\end{equation} Since $V^N_{\alpha,M}(\eta^N,\epsilon)(0)\le K_M|\eta^N(0)-\overline{\eta}^{N, N\epsilon}(0)|$, we use translational invariance to obtain \begin{equation}
		\label{eq: reassembly 2} \mathbb{E}_{f_N \Pi^N_a}\left[V^N_{\alpha,M}(\eta^N, \epsilon)-\gamma \cH(\eta^N)\right]-\frac \vartheta{2}\chi_N N^{2-d}\d_N(f_N)\le \frac{\lambda}{2}.
	\end{equation} Combining with \eqref{eq: reassembly 0}, we find that, for and any $\gamma, \vartheta, \lambda>0$, for $\epsilon>0$ chosen as above, all $N$ sufficiently large and all $f_N\in G_N(\Pi^N_a, b)$, \begin{equation} \label{eq: supex done} \mathbb{E}_{f_N\Pi^N_a}\left[V^N_\alpha(\eta^N,\epsilon)-\gamma\cH(\eta^N)\right]-\vartheta \chi_N N^{2-d}\d_N(f_N)\le \lambda. \end{equation} Since $\lambda>0$ was arbitrary, the proof of \eqref{eq: supex after FK}, and hence of Theorem \ref{thrm: supex}, is complete.
	 \end{proof}
 
 \section*{Acknowledgments} 

The first author acknowledges support by the Max Planck Society through the Research Group ``Stochastic Analysis in the Sciences (SAiS)'' and the DFG CRC/TRR 388 “Rough Analysis, Stochastic Dynamics and Related Fields”, Project A11. The second author is supported by the Royal Commission for the Exhibition of 1851 and acknowledges support by the European Union (ERC, FluCo, grant agreement No. 101088488), as well as the hospitality of the group ``Pattern Formation, Energy Landscapes and Scaling Laws'' of the Max Planck Institute for Mathematics in the Sciences. Views and opinions expressed are however those of the author(s) only and do not necessarily reflect those of the European Union or of the European Research Council. Neither the European Union nor the granting authority can be held responsible for them.  
 
\begin{appendix}

			\begin{appendix} \section{Some Facts on Orlicz Norms} \label{sec: Orlicz Appendix} 
		
		\subsection{Proof of Lemma \ref{lemma: orlicz consistency}} \label{subsec: orlicz consistency} Fix $\Phi, p, \cP_N, \widetilde{\cP}_N, w^N$ as in the lemma, let $h: \cP_N\to \RR$ and let $\widetilde{h}$ be as defined in the statement. Using the definitions of the Orlicz norms, $\widetilde{h}$ and the Lebesgue norms, we write \begin{equation}
			\begin{split} \label{eq: orlicz cons definition chasing}
				\|\widetilde{h}(\widetilde{B})\|_{\ell_\Phi(\widetilde{B}\in \widetilde{\cP}_N,w^N)}&=\inf\left\{t\ge 0: \sum_{\widetilde{B}\in \widetilde{\cP}_N} \frac{w^N(\widetilde{B})}{w^N(\TND)}\Phi\left(\frac{\|h\|_{\ell^p, \cP_N|_{\widetilde{B}}}}{t}\right) \le 1\right\}\\ &=\inf\left\{t\ge 0: \sum_{\widetilde{B}\in \widetilde{\cP}_N} \frac{w^N(\widetilde{B})}{w^N(\TND)}\Phi\left(\left\|{h}/{t}\right\|_{\ell^p, \cP_N|_{\widetilde{B}}}\right) \le 1\right\} \\ &=\inf\left\{t\ge 0: \sum_{\widetilde{B}\in \widetilde{\cP}_N} \frac{w^N(\widetilde{B})}{w^N(\TND)}\Phi\left(\left[\sum_{B\in \cP_N|_{\widetilde{B}}} \frac{w^N(B)}{w^N(\widetilde{B})} \left(\frac{h(B)}{t}\right)^p \right]^{1/p}\right) \le 1\right\}.
			\end{split}
		\end{equation} Recalling that $\Theta(u):=(\Phi^{-1}(u))^p$ is assumed to be convex, we can rewrite the argument of $\Phi$ in the final expression as \begin{equation}
			\left[\sum_{B\in \cP_N|_{\widetilde{B}}} \frac{w^N(B)}{w^N(\widetilde{B})} \left(\frac{h(B)}{t}\right)^p \right]^{1/p}=\Phi^{-1}\left(\Theta^{-1}\left[\sum_{B\in \cP_N|_{\widetilde{B}}} \frac{w^N(B)}{w^N(\widetilde{B})} \Theta\left(\Phi\left[\frac{h(B)}{t}\right]\right)\right]\right).
		\end{equation} Since $\Theta$ is assumed to be convex, we may bound this expression below by Jensen's inequality to find \begin{equation}
			\left[\sum_{B\in \cP_N|_{\widetilde{B}}} \frac{w^N(B)}{w^N(\widetilde{B})} \left(\frac{h(B)}{t}\right)^p \right]^{1/p}\ge \Phi^{-1}\left(\sum_{B\in \cP_N|_{\widetilde{B}}} \frac{w^N(B)}{w^N(\widetilde{B})} \Phi\left[\frac{h(B)}{t}\right]\right).
		\end{equation} Returning to \eqref{eq: orlicz cons definition chasing}, it follows that \begin{equation} \begin{split}
			\|\widetilde{h}(\widetilde{B})\|_{\ell_\Phi(\widetilde{B}\in\widetilde{\cP}_N,w^N)} \ge \inf\left\{t\ge 0: \sum_{\widetilde{B}\in \widetilde{\cP}_N} \frac{w^N(\widetilde{B})}{w^N(\TND)} \sum_{B\in \cP_N|_{\widetilde{B}}} \frac{w^N(B)}{w^N(\widetilde{B})}\Phi\left(\frac{h(B)}t\right)\le 1\right\}. \end{split}
		\end{equation} The double summation collapses to a single summation over $B\in \cP_N$, and this is the definition of $\|h(B)\|_{\ell_\Phi(B\in \cP_N,w^N)}$. \subsection{Proof of Lemma {lemma: orlicz interpolation}}\label{subsec: orlicz interpolation} We now prove the interpolation in Lemma \ref{lemma: orlicz interpolation}. 
			Let $\Psi:[0,\infty)\to [0,\infty)$ be the convex dual of $\Phi$. Using H\"older's inequality for Orlicz norms \cite[Proposition 1.14]{leonard2007orlicz}, for any functions $f, g$ it holds that \begin{equation}
				\frac1{N^d}\sum_{x\in \TND} |f(x)g(x)|\le 2\|f\|_{\ell_\Phi}\|g\|_{\ell_\Psi}.
			\end{equation} We apply this inequality with $f=u^\alpha, g=1(u>k)$, for some $k\in (0,\infty)$ to be chosen later. We thus find that, for any $u$ with $\|u\|_{\ell^1}\le b$, \begin{equation}
				\|u^\alpha1(u>k)\|_{\ell^1}\le \frac{2\|u^\alpha\|_{\ell_\Phi}}{\Psi^{-1}(1/\|1(u>k)\|_{\ell^1})}\le \frac{2\|u^\alpha\|_{\ell_\Phi}}{\Psi^{-1}(k/b)}
			\end{equation} where the second line uses Chebychev's inequality in the denominator. Including the contributions at small values, \begin{equation}
				\|u^\alpha\|_{\ell^1}\le k^\alpha+ \frac{2\|u^\alpha\|_{\ell_\Phi}}{\Psi^{-1}(k/b)}.
			\end{equation} Setting $k=b\Psi(2/\delta)$ makes the factor in the second term $\delta$, and we are done. \subsection{Technical Constructions}\label{subsec: orlicz constructions} We finally give the constructions of a Young function and weight deferred from Proposition \ref{prop: Young function} and Lemma \ref{lemma: p and w}. \begin{proof}[Proof of Proposition \ref{prop: Young function}] Let $\chi_N\to 0$ be any sequence, and fix $p>1, \delta>0$. The task is to construct a strict Young function $\Phi$ such that the dual function $\Psi$ satisfies \begin{equation}
				\label{eq: choose Phi reiterate} \sup_{u\ge 1}u^{-1/d}\Psi^{-1}(u)<\infty; \qquad \sup_N \chi_N^{\delta/2}\Psi^{-1}(N^d)<\infty			\end{equation} and such that \begin{equation}
				\label{eq: theta condition reiterate} \Theta(u):=(\Phi^{-1}(u))^p
			\end{equation} is also convex. Observe first that if \eqref{eq: choose Phi reiterate} holds for a Young function $\Phi$ and $\widehat{\Phi}$ is another Young function with the pointwise inequality $\widehat{\Phi}\le \Phi+C$, then \eqref{eq: choose Phi reiterate} automatically also holds for $\widehat{\Phi}$, since each operation of the convex dualisation and functional inverse reverse the ordering of pointwise inequality to yield $\widehat{\Psi}^{-1}(u)\le \Psi^{-1}(u+C)$ and concavity of $\Psi^{-1}$ yields $\Psi^{-1}(u+C)\le \Psi^{-1}(u)+C_2$ for some finite $C_2$, uniformly in $u\ge 1$, which implies that $\widehat{\Psi}$ inherits the property \eqref{eq: choose Phi reiterate} from the same property of $\Psi$.\\ We thus show in a first step how a strict Young function satisfying \eqref{eq: choose Phi reiterate} may be constructed, and in a second step how, given any strict Young function $\Phi$, another strict Young function $\widehat{\Phi}\le\Phi+C$ may be chosen to impose \eqref{eq: theta condition reiterate}. \medskip \\ \paragraph{\em Step 1. Constructing $\Phi$ to impose \eqref{eq: choose Phi reiterate}} Since the condition \eqref{eq: choose Phi reiterate} only involves $\Psi^{-1}$, it is convenient to construct a strict, strictly increasing Young function $\Psi$ whose inverse function $\Psi^{-1}$ satisfies both inequalities, and define $\Phi$ implicitly by taking it to be the convex dual of $\Psi$. \\ \\ Set $\Psi_1(x):=|x|^d$, and construct $\Psi_2$ as follows. Let $\{x_n, n\ge 0\}$ be a strictly increasing enumeration of the set $\{0\}\cup\{\chi_N^{-\delta/2}: N\ge 1\}$, which is possible because $\chi_N\to 0$ and the given set is locally finite. Set $$ y_0=0; \qquad y_n:= \max\{N^d: \chi_N^{-\delta/2}=x_n\}, n\ge 1. $$ We now define $\Psi_2$ on $[0,\infty)$ by setting $\Psi_2(0)=0$ specifying a weak derivative \begin{equation}
				\label{eq: I'm basically a wizard anyway} \Psi_2'(x):= \max\left\{\frac{y_{n+1}-y_n}{x_{n+1}-x_n}: n\ge 0, x_n\le x\right\}.
			\end{equation} The postulated expression is nondecreasing and locally finite, so that this uniquely specifies a convex function $\Psi_2$, and a straightforward induction shows that $\Psi_2(x_n)\ge y_n$ for all $n$. Extending $\Psi_2$ to an even function on the whole line $\RR$, we now take $\Psi(x):=\Psi_1+\Psi_2+x^2$, which is convex, continuous, everywhere finite and strictly increasing on $[0,\infty)$, because the same is true for all of $\Psi_1, \Psi_2, x^2$, and $\Psi(x)/x\ge x \to \infty$ so $\Psi$ is strict. The first condition in \eqref{eq: choose Phi reiterate} holds because $\Psi(x)\ge \Psi_1(x)=x^d$, and the second because, for each $N$, there exists $n$ such that $\chi_N^{-\delta/2}=x_n$, and by construction of the $y$s, $N^d \le y_n$. As a result,$$ \Psi(\chi_N^{-\delta/2})\ge \Psi_2(\chi_N^{-\delta/2})=\Psi_2(x_n)\ge y_n\ge N^d$$ which implies that $\Psi^{-1}(N^d)\le \chi_N^{-\delta/2}$.\\ \paragraph{\em Step 2. Imposition of the Convexity Condition \eqref{eq: theta condition reiterate}} Let us now assume that we are given a strict Young function $\Phi$ and an index $p>1$; the task is to construct a new strict Young function $\widehat{\Phi}$ satisfying, for some constant $C$ and all $u\in \mathbb{R}$, $\widehat{\Phi}(u)\le \Phi(u)+C$, and such that \eqref{eq: theta condition reiterate} holds for $\widehat{\Phi}$. \\ \\ Since $\Phi$ is convex and globally defined, it admits a locally bounded, nondecreasing weak derivative $\Phi'$; let $\nu$ be the locally finite Borel measure on $[0,\infty)$ given by $\nu([0,x])=(\Phi'(x)-a)_+$, with $a:=\Phi'(1)$. Setting $g$ to be a smooth mollifier with $g\ge 0, \int_{-1}^{0} g dy =1$ and $\text{supp}(g)=[-1,0]$, set $$f(x):=(\nu \star g)(x)=\int_{[0,\infty)} g(x-y)\nu(dy)$$ and let $F(x):=a+\int_0^x f(x')dx'.$ We now consider the ODE $\psi'(x)=G(x,\psi(x))$ in $x\ge 1$, where $G$ is the function $$  G(x,y)=\tanh\left(\frac{p-1}{2x}\left(\frac{F^2}{f+1}\right)(y)\right) $$ with initial condition $\psi(1)=1.$ A simple calculus exercise shows that $G$ is locally Lipschitz in both arguments $x,y$, which produces unique maximal solutions, and the bound $G(x,y)\le 1$ implies that the maximal solution is in fact global. This ODE therefore uniquely specifies a smooth, strictly increasing map $\psi:[0,\infty)\to [0,\infty)$, and by a straightforward argument by contradiction, $\psi(x)\to \infty$ as $x\to \infty$. We now define $\widehat{F}$ by \begin{equation}
				\widehat{F}(x):=\begin{cases}
					a, & x\in [0,1]; \\ F(\psi(x)), & x\ge 1
				\end{cases}
			\end{equation} so that $\widehat{F}$ is nondecreasing, nonnegative and locally bounded. As a result, there is a unique Young function $\widehat{\Phi}$ with $\widehat{\Phi}(0)=0$ and $\widehat{\Phi}'(u)=\widehat{F}(u)$ for $u> 0$. Since $\Phi$ is strict, $F\to \infty$ as $x\to \infty$ and so $\widehat{F}\to \infty$ as $x\to \infty$, implying that $\widehat{\Phi}$ is also strict. Since $G\le 1$, $\psi(x)\le x$ for all $x\ge 1$, so $\widehat{F}(x)\le F(x)$ from which it is an easy deduction that $\widehat{F}(x)\le \Phi'(x)$ for all $x\ge 1$. As a result, $\widehat{\Phi}(x)\le \Phi(x)+(a-\Phi(1))_+$, which is an inequality of the desired form.\\ \\ It remains to be shown how the construction imposes the condition \eqref{eq: theta condition reiterate}. An exercise in calculus shows that \eqref{eq: theta condition reiterate} is equivalent to imposing, for all $x\ge 0$, \begin{equation} \label{eq: theta condition reiterate reiterate}
				(p-1)(\widehat{\Phi}')^2(x)\ge x\widehat{\Phi}''(x)
			\end{equation} where we observe that $\widehat{\Phi}$ has a second weak derivative, given by $\widehat{\Phi}''(x)=\psi'(x)f(\psi(x))1(x>1)$. The reformulated condition \eqref{eq: theta condition reiterate reiterate} holds automatically on the interval $[0,1)$ on which $\widehat{\Phi}$ is affine. For $x\ge 1$, we use the identification of $\widehat{\Phi}'(x)=\widehat{F}(x)$ and the second derivative above, the ODE and the function $G$ defining $\psi$, and the observation that $\tanh(x)\le x$ to verify \eqref{eq: theta condition reiterate reiterate}, which completes the proof of the claimed properties.     \end{proof} Finally we give the proof of a suitable lattice partitions and weighting. \begin{proof}[Proof of Lemma \ref{lemma: p and w}] Let us fix $N, \chi_N$. For ease of presentation, we divide the proof into steps. It will be efficient to give the decompositions of the torus first in dimension 1 in Steps 1-2, and show a single expression for how this may be tensorised into any dimension in Step 3.  \paragraph{\em Step 1. Nearly Dyadic Decomposition of the 1-dimensional Lattice} Choose $m:=\lfloor \log_2 N\rfloor$, so that $2^m\le N<2^{m+1}$, and set $M=2^m, h=N-M\ge 0$. We set, for all $1\le i\le d$, \begin{equation} \label{eq: basic partition}
				\cP^{1,i}_N=\left\{\left\{\frac{2x}{N}, \frac{2x+1}{N}\right\}: 0\le x<h\right\}\cup\left\{\frac{x}{N}, 2h\le x<N\right\}
			\end{equation} so that $\cP^{1,i}_N$ decomposes the rescaled torus $\mathbb{T}^1_N=\{0, N^{-1}, 2N^{-1},\dots 1-N^{-1}\}$ into $M$ sets, each of size at most 2. \paragraph{\em Step 2. Multiscale Decomposition in Dimension 1.} Given $\delta>0$ fixed, set $q:=\lfloor -\delta \log_2 \min(\chi_N, 1/2)/8 \rfloor$, and choose $K_N, r\in \mathbb{N}, r\in [0, k)$ so that $m=m(N)$ defined in Step 1 can be written $m=qK_N+r$. We now define $$ {l}^k_N:=2^{q(k-1)}, \quad 1\le k < K_N; \qquad \widehat{l}^k_N:= 2^m. $$ For the first property in \eqref{eq: consistency of scales}, we have $$\frac{{l}^{k+1}_N}{{l}^k_N} \le 2^{2q}\le \chi_N^{-\delta/4}$$ for $k<K_N$,  while for the second, $$ \frac{{l}^{k+1}_N}{{l}^k_N} \ge 2^{q}\ge  \frac12 \min\left(\chi_N, \frac12\right)^{-\delta/8}. $$  We now let $\cP^{k, i}_N$, $1<k\le K_N, 1\le i\le d$ be the coarsening of $\cP^{1,i}_N$ defined in \eqref{eq: basic partition} by grouping $l^k_N$ adjacent elements together: enumerating the blocks of $\cP^{1,i}_N$ in increasing order by $A^{1,i}, 1\le i\le M$, define \begin{equation}
				\label{eq: coarse partition} \cP^{k,i}_N:=\left\{A^{1,tl^k_N+1}\cup \dots \cup A^{1,(t+1)l^k_N}, \quad 0\le t\le \frac{M}{l^k_N}\right\}.
			\end{equation} Since each block of $\cP^{1,i}_N$ is of size either 1 or 2, each block in $\cP^{k,i}_N$ has side length between $l^k_N$ and $2l^k_N$, and $\cP^{K_N, i}_N=\mathbb{T}_N$. Moreover, $\cP^{k,i}_N\prec \cP^{k+1,i}_N$, since any block of $\cP^{k+1,i}_N$ is formed by grouping $2^q$ consecutive blocks of $\cP^{k,i}_N$. \paragraph{\em Step 3. Partitions in any dimension} The procedure for converting such a decomposition in $d=1$ to the higher-dimensional case $d\ge 2$ is the same at every scale, and so we write down a single expression. For $1\le k\le K_N$, define \begin{equation}
				\cP^k_N:=\left\{\prod_{i=1}^d A^{k,i}: A^{k,i}\in \cP^{k, i}_N\right\}.
			\end{equation} Every side length of every block is between $l^k_N$ and $2l^k_N$ as observed in Step 2, so we have $\cP^k_N\in \Q(N, l^k_N)$ as desired. The properties that $\cP^k_N\prec \cP^{k+1}_N$ and $\cP^{K_N}_N=\{\TND\}$ are inherited directly from the same properties in dimension 1.      \paragraph{\em Step N. Definition of the weight.} For $x\in \TND$, let $A^1(x)$ be the element of $\cP^1_N$ containing $x$, and set $$ w^N(x):=\frac1{\#A^1(x)}.$$ Since every block of $\cP^1_N$ has side lengths in $\{1, 2\}$, $1\le \#A^1(x) \le 2^d$, so $\sup_x w^N(x)/\inf_x w^N(x)\le 2^d$ is bounded independently of $N$. Meanwhile, every element $A\in \cP^1_N$ receives weight $w^N(A)=1$, and any element $B$ of $\cP^k_N$ with $1\le k<K_N$ is formed by joining exactly $2^{qd(k-1)}$ blocks of $\cP^1_N$, so has weight $w^N(B)=2^{qd(k-1)}$. The statement that all blocks of $\cP^{K_N}_N$ receive the same weight is empty, because $\cP^{K_N}_N$ is a singleton.  		\end{proof}   \end{appendix}

 \end{appendix}

\bibliographystyle{plain}
\bibliography{literature}

\end{document}